\documentclass[smallcondensed]{svjour3}      
\smartqed  
\usepackage{graphicx}

\usepackage{amssymb,amsmath, color}
\usepackage{amssymb}
\usepackage{amsbsy}
\usepackage{enumitem,algorithm2e,algorithmic}

\usepackage{ulem}

\usepackage[colorlinks=true]{hyperref}

\newcommand{\R}{{\mathbb R}}

\DeclareMathOperator{\argmin}{argmin}
\DeclareMathOperator{\prox}{prox}

\newcommand{\cD}{{\mathcal D}}
\newcommand{\cG}{{\mathcal G}}
\newcommand{\cC}{{\mathcal C}}

\newcommand{\cH}{{\mathcal H}}

\newcommand{\demi}{\frac{1}{2}}
\newcommand{\ie}{{\it i.e.}\,\,}

\newlength{\textlarg} 

\newcommand{\eqdef}{:=}

\newcommand{\pa}[1]{\left({#1}\right)}

\newcommand{\interior}{{\rm int}\kern 0.06em}
\newcommand{\inte}{{\rm int}\kern 0.06em}
\newcommand{\cl}{{\rm cl}\kern 0.06em}
\newcommand{\zer}{{\rm zer}\kern 0.06em}
\newcommand{\gph}{{\rm gph}\kern 0.06em}
\newcommand{\dom}{{\rm dom}\kern 0.06em}
\newcommand{\pr}{{\rm pr}\kern 0.06em}
\newcommand{\e}{\varepsilon}

\def\d{\delta}

\def\<{\langle}
\def\>{\rangle}

\usepackage{epsfig}
\usepackage{geometry}
 \usepackage{pict2e}

\if
 {
 \usepackage[pageref]{backref}
\renewcommand*{\backrefalt}[4]{%
\ifcase #1 %
(Not cited)%
\or
(Cited on p.~#2)%
\else
(Cited on pp.~#2)%
\fi
}

}
\fi

\usepackage{ulem}

\begin{document}

\title{Accelerated gradient methods combining Tikhonov regularization 
with geometric damping driven by the Hessian}

\titlerunning{Inertial gradient dynamics with Tikhonov regularization}

\author{Hedy ATTOUCH \and A\"icha BALHAG   \and Zaki CHBANI   \and Hassan RIAHI}

\institute{
Hedy ATTOUCH  \at IMAG, Univ. Montpellier, CNRS, Montpellier, France\\
hedy.attouch@umontpellier.fr,
\and
A\"icha BALHAG  \at Institut de Math\'ematiques de Bourgogne, UMR 5584 CNRS, Universit\'e Bourgogne Franche-Comt\'e, F-2100 Dijon, France\\
aichabalhag@gmail.com 
\and Zaki CHBANI   \and Hassan RIAHI\\
 Cadi Ayyad University \\ S\'emlalia Faculty of Sciences 
 40000 Marrakech, Morocco\\
   chbaniz@uca.ac.ma  \and h-riahi@uca.ac.ma 
}
\maketitle


\begin{abstract}
  In a Hilbert  setting,  for convex differentiable optimization, we consider accelerated gradient dynamics combining Tikhonov regularization with Hessian-driven damping.
The Tikhonov regularization parameter is assumed to tend to zero as time tends to infinity, which preserves equilibria. 
 The presence of the Tikhonov regularization term induces  a strong convexity property which  vanishes asymptotically. 
 To take advantage of the exponential convergence rates attached to  the heavy ball method in the strongly convex case, we consider the inertial dynamic where  the viscous  damping coefficient  is taken proportional to the square root of the Tikhonov regularization parameter, and therefore  also converges towards zero. Moreover, the dynamic involves a geometric damping which is driven by the Hessian of the function to be  minimized, which induces a significant attenuation of the oscillations.
Under an appropriate tuning of the parameters, based on Lyapunov's analysis, we show that the trajectories have at the same time several remarkable properties: they provide fast convergence of  values, fast convergence of  gradients towards zero, and  
strong convergence  to the minimum norm  minimizer.
This study extends a previous paper by the authors where similar issues were examined but without the presence of  Hessian driven damping. 
\end{abstract}

\medskip

\keywords{Accelerated gradient methods; convex optimization; damped inertial dynamics; Hessian-driven damping; hierarchical minimization;    Nesterov accelerated gradient method; Tikhonov approximation.}

\medskip

\subclass{37N40, 46N10, 49M30, 65K05, 65K10, 65K15, 65L08, 65L09, 90B50, 90C25.}


\vspace{5mm}

\section{Introduction}

Throughout the paper, $\mathcal H$ is a real Hilbert space which is endowed with the scalar product $\langle \cdot,\cdot\rangle$, with $\|x\|^2= \langle x,x\rangle    $ for  $x\in \mathcal H$.
Given $f : \mathcal H \rightarrow \mathbb R$  a general convex function, which is continuously differentiable,
we will develop fast gradient methods for solving
 the  minimization problem
\begin{equation}\label{edo0001}
 \min \left\lbrace  f (x) : \ x \in \mathcal H \right\rbrace.
\end{equation}
Our approach is based on  the convergence properties as $t \to +\infty$ of the trajectories generated by the damped inertial dynamic 
\begin{equation*}
{\rm(TRISH)} \qquad \ddot{x}(t) + \delta \sqrt{\varepsilon(t)}  \dot{x}(t) +  \beta \nabla^{2}f(x(t))\dot{x}(t)+  \nabla f (x(t)) + \varepsilon (t) x(t) =0,
\end{equation*}
and on the link between dynamical systems and the algorithms that result from their temporal discretization.
We use (TRISH) as  shorthand for Tikhonov regularized inertial system with Hessian-driven damping.
As a basic ingredient, this system involves 
  a nonnegative function $\varepsilon (\cdot)$ which enters both in the viscous damping and the Tikhonov regularization terms.
We assume that $\lim_{t\rightarrow +\infty} \varepsilon(t) = 0$, which preserves the equilibria. According to the structure of (TRISH) this makes the damping coefficient asymptotically vanish, in coordination with  the Tikhonov regularization coefficient.
The other basic ingredient is the Hessian driven damping term which induces several favorable properties, notably a significant reduction of the oscillations.

\noindent We will show that a judicious setting of $\varepsilon (\cdot)$ and of the positive parameter $\delta$ ensures that the trajectories generated by (TRISH) verify the following three properties at the same time: 

\smallskip

\noindent \; $\bullet$ rapid convergence of values (one can approach arbitrarily close to the optimal convergence rate),

\noindent \; $\bullet$ rapid convergence of the gradients towards zero, 

\noindent \; $\bullet$  strong convergence towards the  minimum norm element of $S= \argmin f$.

\smallskip
 
\noindent Throughout the paper, we assume that the objective function  $f$ and the Tikhonov regularization parameter $\varepsilon (\cdot)$ satisfy the following hypothesis: 
\begin{align*}
( \mathcal{A}) \;\begin{cases}
 \; \; f : \mathcal H \rightarrow \mathbb R \mbox{ is convex, of class } \mathcal C^2,  \nabla f \mbox{ is Lipschitz continuous on  bounded sets}; \vspace{1mm} \\
 \; \; S := \mbox{argmin}_{\cH} f \neq \emptyset. \mbox{ We denote by } x^*  \mbox{ the element of minimum norm of } S;   \vspace{1mm}\\
\; \;  \varepsilon : [t_0 , +\infty [ \to \mathbb R^+  \mbox{ is   a nonincreasing function, of class } \mathcal C^1, \mbox{ such that }\  \lim_{t \to \infty} \varepsilon (t) =0.
\end{cases}
\end{align*}

 We will explain at the end of the article how our study can be extended to the case of a convex lower semicontinuous proper function $f: \cH \to \R \cup \left\lbrace +\infty \right\rbrace$, and give existence and uniqueness results for the Cauchy problem associated with our dynamics.

\subsection{The role of the Tikhonov regularization}
Initially designed for the regularization of ill-posed inverse problems \cite{Tikh,TA}, the  field of application of the Tikhonov regularization was then considerably widened. 
The coupling of first-order in time  gradient systems with a Tikhonov approximation whose coefficient tends asymptotically towards zero has been highlighted in a series of papers  \cite{AlvCab}, \cite{Att2},   \cite{AttCom}, \cite{AttCza2}, \cite{BaiCom}, \cite{Cab}, \cite{CPS}, \cite{Hirstoaga}. 
Our approach builds on several previous works that have paved the way
concerning the coupling of damped second-order in time gradient systems with Tikhonov approximation. First studies  concerned  the heavy ball with friction system of Polyak \cite{Polyak},
where the damping coefficient $\gamma >0$ is  fixed. In   \cite{AttCza1} Attouch and Czarnecki considered the  system
\begin{equation}\label{HBF-Tikh}
 \ddot{x}(t) + \gamma \dot{x}(t) + \nabla f(x(t)) + \varepsilon (t) x(t) =0.
\end{equation}
In the slow parametrization case $\int_0^{+\infty} \varepsilon (t) dt = + \infty$, they proved that  any solution $x(\cdot)$ of \eqref{HBF-Tikh} converges strongly to the minimum norm element of $\argmin f$, see also \cite{Att-Czar-last}, \cite{Cabot-inertiel}, \cite{CEG},  \cite{JM-Tikh}. This hierarchical minimization result contrasts with the case without the Tikhonov regularization term, where the convergence holds only for weak convergence, and the limit depends on the initial data. 

\noindent In the quest for a faster convergence, the following system
with asymptotically vanishing damping
\begin{equation}\label{edo001-0}
 \mbox{(AVD)}_{\alpha, \varepsilon} \quad \quad \ddot{x}(t) + \frac{\alpha}{t} \dot{x}(t) + \nabla f (x(t)) +\varepsilon(t) x(t)=0,
\end{equation}
was studied by Attouch, Chbani, and Riahi in \cite{ACR}.
It is a Tikhonov regularization of the  dynamic
\begin{equation}\label{edo001}
 \mbox{(AVD)}_{\alpha} \quad \quad \ddot{x}(t) + \frac{\alpha}{t} \dot{x}(t) + \nabla f (x(t))=0,
\end{equation}
which was introduced by  Su, Boyd and
Cand\`es in \cite{SBC}. $\mbox{(AVD)}_{\alpha}$ is a low resolution ODE of the   accelerated gradient method of Nesterov \cite{Nest1,Nest2} and  of the Ravine method \cite{AF}, \cite{SBC}.
$ \mbox{(AVD)}_{\alpha}$ has been the subject of many recent studies which have given an in-depth understanding of the Nesterov accelerated gradient method, see  \cite{AAD1}, \cite{ABCR}, \cite{AC10}, \cite{ACPR},\cite{AP}, \cite{CD}, \cite{MME}, \cite{SBC}, \cite{Siegel}, \cite{WRJ}.

As an original aspect of our approach, we rely on  the properties of the heavy ball with friction method of Polyak in the
\textit{strongly convex case}, which 
provides exponential convergence rates. To take advantage of this remarkable property, and adapt it to our situation, we consider the nonautonomous dynamic version of  the heavy ball method which at time $t$ is governed by the gradient of the regularized function
$x\mapsto f(x) + \frac{\varepsilon (t)}{2}\|x\|^2$, where the Tikhonov regularization parameter satisfies $\varepsilon (t) \to 0$ as $t\to +\infty$. This idea was first developed in \cite{ABCR}, \cite{AL}. Let us make this precise.
 
 Recall that
a function  $f: \cH \to \mathbb R$ is $\mu$-strongly convex for some $\mu >0$ if   $f- \frac{\mu}{2}\| \cdot\|^2$ is convex.
In this setting, we have the following exponential convergence result for the heavy ball with friction  dynamic where the viscous damping coefficient is twice the square root of the modulus of strong convexity of $f$, see \cite{Polyak-64}:
\begin{theorem}\label{strong-conv-thm}
Suppose that $f: \cH \to \mathbb R$ is a function of class ${\mathcal C}^1$ which is $\mu$-strongly convex for some $\mu >0$.
Let  $x(\cdot): [t_0, + \infty[ \to \cH$ be a solution trajectory of
\begin{equation}\label{dyn-sc-aa}
\ddot{x}(t) + 2\sqrt{\mu} \dot{x}(t)  + \nabla f (x(t)) = 0.
\end{equation}
 Then, the following property holds: \; 
$
f(x(t))-  \min_{\mathcal H}f  = \mathcal O \left( e^{-\sqrt{\mu}t}\right) \;  \mbox{ as } \; t \to +\infty.$
\end{theorem}
\noindent To adapt this result to the case of a general convex differentiable function $f: \cH \to \mathbb R$, a natural idea  is to use Tikhonov's method of regularization. This leads to consider the  non-autonomous dynamic which at time $t$ is governed by the gradient of the strongly convex function 
 $$\varphi_t: \cH \to \mathbb R, \quad
\varphi_t (x) := f(x) + \frac{\varepsilon(t)}{2} \|x\|^2.
 $$
The viscosity curve $\displaystyle{\varepsilon \mapsto x_{\varepsilon}
\eqdef \argmin_{\cH} \left\lbrace  f(\cdot) + \frac{\varepsilon}{2}\|\cdot \|^2 \right\rbrace}$ will play a key role in our analysis.
By definition of $\varphi_t$, we have 
$
x_{\varepsilon(t)} = \argmin_{\cH}{\varphi}_{t}.
$
 The first-order optimality condition gives 
\begin{equation}\label{opt}
\nabla f(x_{\varepsilon(t)})+\varepsilon(t)x_{\varepsilon(t)}=0.
\end{equation} 
 We call $t \mapsto x_{\varepsilon(t)}$ the parametrized viscosity curve.
  Then, replacing $f$ by $\varphi_t$ in
\eqref{dyn-sc-aa}, and noticing that $\varphi_t$ is $\varepsilon (t)$-strongly convex, this gives the following dynamic which was introduced in \cite{AL} and \cite{ABCR} ($\delta$ is a positive parameter)
\begin{equation*}
{\rm(TRIGS)} \qquad \ddot{x}(t) + \d\sqrt{\e(t)}  \dot{x}(t) + \nabla f (x(t)) + \varepsilon (t) x(t) =0.
\end{equation*}
 (TRIGS) stands shortly for Tikhonov regularization of inertial gradient systems.
In order not to asymptotically modify the equilibria, it is supposed that $\varepsilon (t) \to 0$ as $t\to +\infty$\footnote{This is the key property of the asymptotic version ($t\to +\infty$) of the Browder-Tikhonov regularization method.}. This condition implies that (TRIGS) falls within the framework of the inertial gradient systems with asymptotically vanishing damping.
 It has been shown in  \cite{ABCR}, \cite{AL} that a judicious tuning of  $\varepsilon (t)$ in (TRIGS) ensures both rapid convergence of values, 
 and  strong convergence of the trajectories towards the  minimum norm element of $S= \argmin_{\cH} f$ (which is reminiscent of the Tikhonov method).

 \subsection{The role of the Hessian-driven damping}
As is the case with inertial dynamics which are only damped by viscous damping, the system (TRIGS) may exhibit oscillations
which are undesirable from an optimization point of view. 
To remedy this situation, we introduce into the dynamic a geometric damping which is driven by the Hessian of the function $f$ to be minimized.
So doing, we obtain the  system (TRISH).
The presence of the Hessian does not entail numerical difficulties, since the Hessian intervenes in the above ODE in the form $\nabla^2  f (x(t)) \dot{x} (t)$, which is nothing but the derivative wrt time of $\nabla  f (x(t))$. This explains why the time discretization of this dynamic provides first-order algorithms.
The importance of the Hessian driven damping has been demonstrated in several areas. We list some of them below. 

\smallskip

$\bullet$  In the field of PDEs for mechanics and physics, it is called strong damping, or geometric damping because it takes into account the geometry of the function to be minimized. 
In the PDE's framework, when $f$ is quadratic, and $\nabla f=A$ is a linear elliptic operator, the strong damping involves the action of a fractional power $A^\theta$ of $A$ on the velocity vector. When $\theta \geq \demi$, this induces  notably reduced oscillations. 
The Hessian-driven damping corresponds to $\theta =1$.
It can be combined with various other types of damping, such as the dry friction \cite{AAV-algo}. 
It also makes it possible to model shocks which are completely damped in unilateral mechanics \cite{AMR}.

\smallskip

$\bullet$  It has been shown in \cite{AF} and \cite{SDJS} that the high resolution ODE of the Ravine and Nesterov methods exhibits the Hessian driven damping. This explains the rapid convergence of the gradients towards zero which is verified by these dynamics and algorithms \cite{ACFR}, \cite{ACFR-Optimisation}, \cite{APR}, 
\cite{BCL}, \cite{SDJS}.
Our approach is in accordance with Nesterov \cite{Nest3}, where it is conjectured that the introduction of an adapted Tikhonov regularization term
helps to make the gradients small.

\smallskip

$\bullet$   The Hessian driven damping  comes into the study of Newton's method in optimization.
Given a general maximally monotone operator $A: \cH \to 2^{\cH}$,
to overcome the ill-posedness of  Newton's   continuous method for solving $0 \in A(x)$, 
the following first-order evolution system  was considered by Attouch and Svaiter \cite{AS} and  studied further in \cite{AAS}, \cite{AMAS}. Formally, this system is written as

\smallskip

\begin{center}
$
\gamma(t)  \dot{x}(t) + \beta   \frac{d}{dt} \left( A(x(t))\right)   + A(x(t)) = 0.
$
\end{center}


\noindent It can be considered as a continuous version of the Levenberg-Marquardt method, which acts as a regularization of the Newton method.
Under a fairly general assumption on the regularization parameter $\gamma (\cdot)$, this system is well posed and generates trajectories that converge weakly to equilibria.
Thus, (TRISH) and its nonsmooth extension can be considered as an inertial and regularized version of this  system when $A$ is the subdifferential of a convex lower semicontinuous proper function.

 \subsection{A model result}
 
In section \ref{sec:particular-cases}, we will prove the following result in the case $\varepsilon (t) = \frac{1}{t^r}$. It is expressed with the help  of the parametrized viscosity curve which converges strongly to the minimum norm solution.

\begin{theorem}\label{thm:model-intro}
Take   $0<r<2$, \; $\delta>2$, \; $\beta >0$.
Let $x : [t_0, +\infty[ \to \mathcal{H}$ be a solution trajectory of
		\begin{equation}\label{particular-intro}
		\ddot{x}(t) + \frac{\d}{ \displaystyle{t^{\frac{r}{2}}}}\dot{x}(t) +\beta\nabla^{2} f\left(x(t) \right)\dot{x}(t)+ \nabla f\left(x(t) \right)+ \frac{1}{t^r} x(t)=0.
		\end{equation}
	Then, we have fast convergence of the values, fast convergence of the gradients towards zero, 
	and  strong convergence of the trajectory to the minimum norm 	solution, with the following rates: 
	
	\smallskip
			
$\bullet$ \; $f(x(t))-\min_{\cH} f= \mathcal O \left( \displaystyle\frac{1}{t^{r} }   \right)$  \;  as $ t \to +\infty$ ;
 
$\bullet$ \; $\displaystyle{\int_{t_0}^{+\infty}t^{\frac{3r -2}{2}}\Vert\nabla f(x(t))\Vert^2 dt<+\infty }$;

$\bullet$ \; $ \|x(t) -x_{\varepsilon(t)}\|^2=\mathcal{O}\left(\displaystyle{\dfrac{1}{ t^{\frac{2-r}2}}}\right)$ \;  as $ t \to +\infty$.
\end{theorem}
This is the first time that these three properties have been obtained within the same dynamic.
Let us note that by taking $r$ close to $2$, one obtains convergence rates comparable to the most recent results concerning the introduction of the Hessian driven damping in the dynamic associated with the accelerated gradient of Nesterov. Precisely, letting $r\to 2$ in the  formulas above gives  $f(x(t))-\min_{\cH} f= \mathcal O ( 1/t^{2} )$, and  $\int_{t_0}^{+\infty}t^{2}\Vert\nabla f(x(t))\Vert^2 dt<+\infty   $. 
Thus, by taking $r$ sufficiently close to $2$, we can obtain convergence rates arbitrarily close to these rates.
The case $r=2$, which corresponds to the Nesterov accelerated gradient method, is critical: in this case, the strong convergence towards the minimum norm solution is an open question. 
The above results show the balance between   fast convergence of  values and  strong convergence to the minimum norm solution.

\subsection{Contents}
The paper is organized as follows.
In section \ref{sec:Lyap}, for a general Tikhonov regularization parameter $\varepsilon(\cdot)$, we study the asymptotic convergence properties of the solution trajectories of (TRISH). Based on Lyapunov analysis, we show their strong convergence to the  minimum norm element of $S$, and establish the convergence  rates of the  values and integral estimates of the  gradients.
In section \ref{sec:particular-cases}, we apply these results to 
the particular case $\varepsilon (t) =\frac{1}{t^r}$, $0<r<2$, and obtain fast convergence results.
Section \ref{num} contains numerical illustrations. 
Section \ref{sec:nonsmooth} gives indications concerning the extension of our study to the nonsmooth case, and provides existence and uniqueness results for the Cauchy problem associated with the  considered dynamics.
We conclude with a perspective and open questions.

\section{Convergence results via Lyapunov analysis}
\label{sec:Lyap}

Given a general regularization parameter $\varepsilon(\cdot)$,
we successively present the idea  guiding the Lyapunov analysis, then some preparatory lemmas, and finally the detailed proof.
In the next section, we will particularize our results to the case $\varepsilon (t) =\frac{1}{t^r}$, $0<r<2$, and obtain fast convergence results.
\subsection{General idea of the proof}
  As we already mentioned,  the function
\begin{equation}\label{def:phi}
\varphi_t: \cH \to \mathbb R, \quad
\varphi_t (x) := f(x) + \frac{\varepsilon(t)}{2} \|x\|^2
\end{equation}
plays a central role in the Lyapunov analysis, via its strong convexity property.
Thus, it is convenient to reformulate (TRISH) with the help of the function $\varphi_t$, which gives
\begin{equation}\label{1sans}
{\rm(TRISH)} \qquad \ddot{x}(t)+\delta\sqrt{\varepsilon(t)}\dot{x}(t)+\beta \nabla^{2}f(x(t))\dot{x}(t)+\nabla {\varphi}_{t}(x(t))=0,
\end{equation}
where    $\delta,\beta$ are positive parameters. We recall that  $\varepsilon : [t_0 , +\infty [ \to \mathbb R^+  $ is a nonincreasing
function of class $\mathcal{C}^{1}$, such that $\lim_{t\rightarrow +\infty} \varepsilon(t) = 0$. 
In the mathematical analysis of  inertial gradient dynamics and algorithms with Hessian-driven damping, the basic equality
\begin{equation}\label{basic_equ_1}
\frac{d}{dt} \nabla f(x(t)) =  \nabla^{2}f(x(t))\dot{x}(t)
\end{equation}
makes these systems relevant to first-order methods, a crucial property for numerical purposes. 
 In the presence of the Tikhonov term, to keep the structural property attached to \eqref{basic_equ_1},
let us introduce the following variant of (TRISH) where the above relation comes with $\varphi_{t}$ instead of $f$:
\begin{equation}\label{1b}
{\rm(TRISHE)} \qquad  \ddot{x}(t)+\delta\sqrt{\varepsilon(t)}\dot{x}(t)+\beta\dfrac{d}{dt}\left(\nabla \varphi_{t}(x(t))\right)+\nabla\varphi_{t}(x(t))=0.
\end{equation}
Adding the  suffix E after TRISH recalls that the dynamic has been adapted to take advantage of the Equality in \eqref{basic_equ_1}, with $\varphi_{t}$ instead of $f$.
To encompass these two dynamic systems, we consider
\begin{equation}\label{1principal}
\ddot{x}(t)+\delta\sqrt{\varepsilon(t)}\dot{x}(t)+\beta\dfrac{d}{dt}\left[\nabla \varphi_{t}(x(t))+(p-1)\varepsilon(t)x(t)\right]+\nabla\varphi_{t}(x(t))=0,
\end{equation}
where the parameter  $p\in [0,1].$
When $p=0$ we get (TRISH), and for $p=1$ we get (TRISHE).

Given $p\in [0,1]$, let us introduce the real-valued function $t \in [t_0, +\infty[ \mapsto E_p(t) \in \R^+$ that plays a key role in our Lyapunov analysis. It is defined by
\begin{equation}\label{3}
E_p(t)\eqdef 
\left(\varphi_{t}(x(t))-\varphi_{t}(x_{\varepsilon(t)})\right) +\dfrac{1}{2}\|v_p(t)\|^{2}
\end{equation}
where $\varphi_{t}$ has been defined in (\ref{def:phi}), $x_{\varepsilon(t)}$ in  \eqref{opt},   and
\begin{equation}\label{3b}
v_p(t)\eqdef \lambda\sqrt{\varepsilon(t)}\left(x(t)-x_{\varepsilon(t)}\right)+\dot{x}(t)+\beta\left[\nabla \varphi_{t}(x(t))+(p-1)\varepsilon(t)x(t)\right],
\end{equation} 
with $0\leq \lambda < \delta$. We will show that under a judicious setting of  parameters,  $ E_p(\cdot)$  satisfies the first-order differential inequality 
\begin{equation}\label{est:basic2}
\dot{E}_p(t)+\mu(t)E_p(t)+\frac{\beta}{2 \delta}\left( \delta -\lambda\right)\Vert\nabla \varphi_t(x(t))\Vert^2 \leq \frac{\Vert x^* \Vert^2}{2}G(t), 
 \end{equation}
 where
 \begin{equation}\label{def:mu}
\mu(t) \eqdef -\dfrac{\dot{\varepsilon}(t)}{2\varepsilon(t)}+(\delta-\lambda)\sqrt{\varepsilon(t)},
\end{equation} 
and
\begin{equation*}
G(t) \eqdef (\lambda c+2a)\lambda\dfrac{\dot{\varepsilon}^{2}(t)}{\varepsilon^{\frac{3}{2}}(t)}-\dot{\varepsilon}(t)+(1-p)\beta \lambda(\delta-\lambda)\varepsilon^{2}(t).
\end{equation*} 
\noindent Since $\mu(t) >0$, this will allow us to estimate the rate of convergence of $E_p(t)$ towards zero. 
In turn, this  provides convergence rates of values and trajectories, as the following lemma shows. 

\begin{lemma}\label{lem-basic-b}
	Let   $x(\cdot): [t_0, + \infty[ \to \cH$ be a solution trajectory of  the damped inertial dynamic \eqref{1principal}, and $t \in [t_0, +\infty[ \mapsto E_p(t) \in \R^+$ be the energy function defined in \eqref{3}. Then, the following estimates are satisfied:   for any $t\geq t_0$, 
	
	\begin{eqnarray}
	&&f(x(t))-\min_{\mathcal H}f	
	\leq  E_p(t)+\dfrac{\varepsilon(t)}{2}\|x^*\|^{2}; \label{keybb-00}\\
	&&\|x(t) - x_{\varepsilon(t)}\|^2  \leq \frac{2E_p(t)}{\varepsilon(t)} \label{est:basic1}.
	\end{eqnarray}
	Therefore, $x(t)$ converges strongly to $x^*$
	as soon as 
	$
	\lim_{t\to +\infty} \displaystyle{\frac{E_p(t)}{\varepsilon(t)}}=0.
	$
\end{lemma}
 \begin{proof}
 
 $i)$
 According to the definition of $ \varphi_{t}$,  we have 
 \begin{equation*}
\begin{array}{lll}
f(x(t))-\min_{\mathcal H}f	& = &  \varphi_{t}(x(t))-\varphi_{t}(x^*)+\dfrac{\varepsilon(t)}{2}\left(\|x^*\|^{2}-\|x(t)\|^{2}\right)  \\ 
	& = & \left[\varphi_{t}(x(t))-\varphi_{t}(x_{\varepsilon(t)})\right]+\left[\underbrace{\varphi_{t}(x_{\varepsilon(t)})-\varphi_{t}(x^*)}_{\leq 0}\right]+\dfrac{\varepsilon(t)}{2}\left(\|x^*\|^{2}-\|x(t)\|^{2}\right)\\
	& \leq  &\varphi_{t}(x(t))-\varphi_{t}(x_{\varepsilon(t)})+\dfrac{\varepsilon(t)}{2}\|x^*\|^{2}.
\end{array}
\end{equation*}
By definition of $E_p(t)$ we have 
 \begin{equation}\label{E_phi}
  \varphi_{t}(x(t))-\varphi_{t}(x_{\varepsilon(t)}) \leq E_p(t)
  \end{equation}
 which, combined with the above inequality, gives \eqref{keybb-00}.

\smallskip

$ii)$
By the strong convexity of $\varphi_{t}$, and 
$x_{\varepsilon(t)}:= \argmin_{\cH}\varphi_{t}$, we have
$$
\varphi_{t}(x(t))-\varphi_{t}(x_{\varepsilon(t)}) \geq \frac{\varepsilon (t)}{2}  \|x(t) - x_{\varepsilon(t)}\|^2 .
$$
By combining the  inequality above with \eqref{E_phi}, we get
$$
E_p(t) \geq \frac{\varepsilon (t)}{2}  \|x(t) - x_{\varepsilon(t)}\|^2 ,
$$
which gives \eqref{est:basic1}.\qed
\end{proof}

\if
{
In order to explore the asymptotic behavior of the systems \eqref{1} and \eqref{1sans}, we  need the following function:
\begin{equation}\label{W}
W_p(t)=e^{\int^{t}_{t_0}\mu(s)ds}E_p(t)\;\text{ where }\mu(t)=-\dfrac{\dot{\varepsilon}(t)}{2\varepsilon(t)}+(\delta-\lambda)\sqrt{\varepsilon(t)}.
\end{equation} 

To compute 
\begin{equation}\label{W1}
\dfrac{d}{dt}W_p(t)=e^{\int^{t}_{t_0}\mu(s)ds}\left[\dfrac{d}{dt}E_p(t)+\mu(t)E_p(t)\right] ,
\end{equation}
}
\fi

\subsection{Preparatory results for Lyapunov analysis}

The parametrized viscosity curve $t\mapsto x_{\varepsilon(t)}$ plays a central role in the definition of $ E_p(\cdot)$, and therefore  in the Lyapunov analysis.
We review below some its topological and differential properties.

\subsubsection{Topological properties}  The following properties  are immediate consequences of  the classical properties of the Tikhonov regularization (see \cite{Att2} for a general overview of viscosity methods), and of  $\lim_{t\to +\infty} \varepsilon (t)=0$:
\begin{eqnarray}
&& \bullet \; \forall t\geq t_0 \;\; \;  \|x_{\varepsilon(t)}\|\leq \|x^{*}\|
\label{2a}  \vspace{6mm}\\
&& \bullet \; \lim_{t\rightarrow +\infty}\|x_{\varepsilon(t)}-x^{*}\|=0 \quad\hbox{where}\: x^{*}=\mbox{proj}_{\argmin f} 0.\label{2b}
\end{eqnarray}

\subsubsection{Differential properties}

To evaluate the terms $\dfrac{d}{dt}\left(\varphi_{t}(x_{\varepsilon(t)})\right)$ and $\dfrac{d}{dt}\left(x_{\varepsilon(t)}\right)$ which occur in $\dfrac{d}{dt}E_p(t)$,
 we use the differentiability properties of the viscosity curve $\epsilon \mapsto x_{\epsilon}= \argmin \left\lbrace f(\xi) + \frac{\epsilon}{2}\|\xi\|^2 \right\rbrace$.
 According to \cite{AttCom}, \cite{Hirstoaga}, \cite{Torralba}, the viscosity curve is  Lipschitz continuous on the compact intervals of $]0, +\infty[$. So it is absolutely continuous, and  almost everywhere differentiable.
Based on these properties we have the following lemma, which was established in \cite{ABCR}, and which we reproduce here for ease of reading.
\begin{lemma}\label{lem1}
The following properties are satisfied:
 \begin{itemize}
 	\item[$i)$] For each $t \geq t_0$, \; 
 	$\dfrac{d}{dt}\left(\varphi_{t}(x_{\varepsilon(t)})\right)=\frac{1}{2}\dot{\varepsilon}(t)\|x_{\varepsilon(t)}\|^{2}$.
 	
 	\smallskip
 	
 	\item[$ii)$] The function $t\mapsto x_{\varepsilon(t)}$ is 
 Lipschitz continuous on the compact intervals of $]t_0, +\infty[$, hence almost everywhere differentiable, and the following inequality holds:  for almost every $t \geq t_0$
 	 $$\left\|\dfrac{d}{dt}\left(x_{\varepsilon(t)}\right)\right\|^{2} \leq -\dfrac{\dot{\varepsilon}(t)}{\varepsilon(t)} \left\langle \dfrac{d}{dt}\left(x_{\varepsilon(t)}\right), x_{\varepsilon(t)}\right\rangle.$$
 \end{itemize}
 	Therefore, for almost every $t \geq t_0$
 	$$\left\|\dfrac{d}{dt}\left(x_{\varepsilon(t)}\right)\right\|\leq -\dfrac{\dot{\varepsilon}(t)}{\varepsilon(t)} \|x_{\varepsilon(t)}\|. $$ 
\end{lemma}

\begin{proof} 

 $i)$  We use some classical differentiability properties of the Moreau envelope. We have    
	$$\varphi_{t}(x_{\varepsilon(t)})=\inf_{\xi\in H} \left\lbrace f(\xi)+\frac{\varepsilon(t)}{2}\|\xi-0\|^{2}\right\rbrace=f_{\frac{1}{\varepsilon(t)}}(0),$$ 
where, for any $\theta >0$, the Moreau envelope $f_{\theta}: \cH \to \R$ is  defined by 
\begin{equation}\label{def:prox-b}
f_{\theta} (x) = \min_{\xi \in \cH} \left\lbrace f (\xi) + \frac{1}{2 \theta} \| x - \xi\| ^2   \right\rbrace .
\end{equation}
Recall that infimum in the above expression is achieved at the unique point
$\prox_{\theta f}(x)$, \ie 
\begin{equation}\label{def:prox-bb}
f_{\theta} (x) = f (\prox_{\theta f}(x)) + \frac{1}{2 \theta} \| x - \prox_{\theta f}(x)\| ^2 . 
\end{equation}
One can consult \cite[section 12.4]{BC} for more details on the Moreau envelope.
	 Since  $ \dfrac{d}{d\theta}f_{\theta}(x)=-\frac{1}{2}\|\nabla f_{\theta}(x)\|^{2}, $ (see  \cite[Appendix, Lemma 3]{ABCR}),    we have:
	$$ \dfrac{d}{dt}f_{\theta(t)}(x)=-\frac{\dot{\theta}(t)}{2}\|\nabla f_{\theta(t)}(x)\|^{2} .$$ 
	Therefore, 
	\begin{equation}\label{4}
	 \dfrac{d}{dt} \varphi_{t}(x_{\varepsilon(t)})=\dfrac{d}{dt}\left(f_{\frac{1}{\varepsilon(t)}}(0)\right)=\frac{1}{2}\dfrac{\dot{\varepsilon}(t)}{\varepsilon^{2}(t)}\|\nabla f_{\frac{1}{\varepsilon(t)}}(0)\|^{2}.
	\end{equation}
On the other hand, we have 
$$ \nabla\varphi_{t}(x_{\varepsilon(t)})=0\Longleftrightarrow \nabla f(x_{\varepsilon(t)})+\varepsilon(t)x_{\varepsilon(t)}=0\Longleftrightarrow x_{\varepsilon(t)}= 
\prox_{ \frac1{\varepsilon (t)} f}(0).
$$	
Since  $\nabla f_{\frac{1}{\varepsilon(t)}}(0)=\varepsilon(t)\left(0-  \prox_{ \frac1{\varepsilon (t)} f}(0)\right), $  we get  
$ \nabla f_{\frac{1}{\varepsilon(t)}}(0)=-\varepsilon(t)x_{\varepsilon(t)} $.  This combined with \eqref{4} gives
$$ \dfrac{d}{dt} \varphi_{t}(x_{\varepsilon(t)})=\frac{1}{2}\dot{\varepsilon}(t)\|x_{\varepsilon(t)}\|^{2}.$$ 
\item[$ii)$] We have 
$$ -\varepsilon(t)x_{\varepsilon(t)}=\nabla f(x_{\varepsilon(t)})\quad\hbox{and}\quad -\varepsilon(t+h)x_{\varepsilon(t+h)}=\nabla f(x_{\varepsilon(t+h)}). $$
According to the monotonicity  of $\nabla f,$ we have
$$ \langle \varepsilon(t)x_{\varepsilon(t)}-\varepsilon(t+h)x_{\varepsilon(t+h)},x_{\varepsilon(t+h)}- x_{\varepsilon(t)}  \rangle \geq 0 ,$$
which implies 
$$ -\varepsilon(t)\|x_{\varepsilon(t+h)}- x_{\varepsilon(t)}\|^{2} + \left(\varepsilon(t)-\varepsilon(t+h)\right) \langle x_{\varepsilon(t+h)},x_{\varepsilon(t+h)}- x_{\varepsilon(t)}  \rangle \geq 0 .$$
After division by $h^{2},$ we obtain 
$$ \dfrac{ \left(\varepsilon(t)-\varepsilon(t+h)\right)}{h} \left\langle x_{\varepsilon(t+h)},\dfrac{x_{\varepsilon(t+h)}- x_{\varepsilon(t)}}{h} \right \rangle \geq \varepsilon(t)\left\|\dfrac{x_{\varepsilon(t+h)}- x_{\varepsilon(t)}}{h}\right\|^{2} .$$
We now rely on the differentiability properties of the viscosity curve $\epsilon \mapsto x_{\epsilon}$, which have been recalled above. It is  Lipschitz continuous on the compact intervals of $]0, +\infty[$, so almost everywhere differentiable. Therefore, the mapping $t \mapsto x_{\epsilon (t)}$ satisfies the same differentiability properties.
 By letting $h \rightarrow 0,$ we obtain that, for almost every $t\geq t_0$
 $$-\dot{\varepsilon}(t) \left\langle x_{\varepsilon(t)},\dfrac{d}{dt}x_{\varepsilon(t)} \right \rangle \geq \varepsilon(t)\left\|\dfrac{d}{dt}x_{\varepsilon(t)} \right\|^{2} ,$$
which gives the claim. The last statement follows from  Cauchy-Schwarz inequality.\qed
\end{proof}

\subsection{Lyapunov analysis of (TRISH): main theorem and its proof}
Take $p\in [0,1]$. In the Lyapunov analysis of
\begin{equation}\label{1principal-bis}
\ddot{x}(t)+\delta\sqrt{\varepsilon(t)}\dot{x}(t)+\beta\dfrac{d}{dt}\left[\nabla \varphi_{t}(x(t))+(p-1)\varepsilon(t)x(t)\right]+\nabla\varphi_{t}(x(t))=0,
\end{equation}
we  assume that the Tikhonov regularization parameter $\varepsilon(\cdot)$ satisfies the following growth condition.
\begin{definition}
 \textit{The Tikhonov regularization parameter
$t\mapsto \varepsilon (t)$ satisfies the condition $( \mathcal{H}_p)$ if there exists $a>1,$  $c>2$, $\lambda >0$ and $t_1 \geq t_0$ such that for all $t \geq t_1 ,$}
	$$( \mathcal{H}_p) \qquad \dfrac{d}{dt}\left(\dfrac{1}{\sqrt{\varepsilon(t)}}\right) \leq 
	\min\left(2\lambda-\delta\; , \; \frac{1}{2}\left(\delta-\frac{a+1}{a}\lambda\right)\right),\quad \mbox{and}\quad \delta\beta\leq\dfrac{1}{\sqrt{\varepsilon(t)}},
	$$
\textit{where, in the above inequality, it is supposed that the parameter $\lambda$ is such that $   \frac{\delta}{2} <     \lambda < \delta$ and satisfies}
	
	\medskip

\noindent $\bullet$\; 
	$ For \; p \notin [0,\frac{1}{2}]  \mbox{ and }\; c>\max\left(2,\frac{1+\sqrt{2(1-p)}}{2p-1}\right)$
	 
\hspace{1cm} $ \frac{\delta}{2}< \lambda < \min\left(\frac{a}{a+1}\delta, \frac{\delta+\sqrt{\delta^{2}-4(1-p)}}{2}\right)\; \mbox{ when } 2\sqrt{1-p} <\delta \leq \sqrt{2}-\frac{1}{c},$ 

\hspace{1cm} $ 
	\frac12\left(\delta+\frac{1}{c}+\sqrt{(\delta+\frac{1}{c})^{2}-2}\right)
	< \lambda < \min\left(\frac{a}{a+1}\delta, \frac{\delta+\sqrt{\delta^{2}-4(1-p)}}{2}\right) \,\mbox{ when } \delta > \sqrt{2}-\frac{1}{c}.$

\noindent $\bullet$   For  $p\in [0,\frac{1}{2}] \quad  \frac12\left(\delta+\frac{1}{c}+\sqrt{(\delta+\frac{1}{c})^{2}-2}\right)
		< \lambda < \min\left(\frac{a}{a+1}\delta, \frac{\delta+\sqrt{\delta^{2}-4(1-p)}}{2}\right) \;  \mbox{ when } \delta > 2\sqrt{(1-p)}.\,$
\end{definition}
\begin{remark}
Integrating the differential inequality $(\mathcal{H}_p)$ shows that  the damping coefficient in \eqref{1principal-bis} (which is proportional to $\sqrt{\e(t)}$) must be greater than or equal to $C/t$ for some positive constant $C$.
 This is consistent with the theory of inertial gradient systems with time-dependent viscosity coefficient, which shows that the asymptotic optimization property is valid provided that the integral of the viscous damping coefficient over $[t_0, +\infty[$ be infinite,  \cite{AC10}, \cite{CEG}. See also \cite{ABCR}, \cite{AL} where a similar growth condition on the Tikhonov parameter is considered.
\end{remark}  
For ease of reading, let us recall  the functions that enter the Lyapunov analysis:
\begin{eqnarray}
&& E_p(t)\eqdef 
\left(\varphi_{t}(x(t))-\varphi_{t}(x_{\varepsilon(t)})\right) +\dfrac{1}{2}\|v_p(t)\|^{2} \label{3bis} \\
&& v_p(t)\eqdef \lambda\sqrt{\varepsilon(t)}\left(x(t)-x_{\varepsilon(t)}\right)+\dot{x}(t)+\beta\left[\nabla \varphi_{t}(x(t))+(p-1)\varepsilon(t)x(t)\right] \label{3bbis}\\
&&\mu(t) :=-\dfrac{\dot{\varepsilon}(t)}{2\varepsilon(t)}+(\delta-\lambda)\sqrt{\varepsilon(t)} \\
&& \gamma(t) \eqdef \exp\left(\displaystyle \int_{t_1}^{t} \mu(s)ds\right). \label{def:mu_gamma}
 \end{eqnarray}
We can now state our main convergence result.
\begin{theorem}\label{strong-conv-thm-b}
	Let  $x(\cdot): [t_0, + \infty[ \to \cH$ be a solution trajectory of the system \eqref{1principal-bis}. Take $\delta>2\sqrt{(1-p)}$.	Suppose that 
 $\varepsilon (\cdot)$ satisfies the condition $( \mathcal{H}_p)$.	Then,  the following properties are satisfied:  for all $ t\geq t_1$
	\begin{eqnarray}
	&& E_p(t)\leq \dfrac{\|x^{*}\|^{2}}{2\gamma(t)}\displaystyle \int_{t_1}^{t}G(s)\gamma(s) ds+ \dfrac{\gamma(t_1)E_p(t_1)}{\gamma(t)}   ,  \label{Lyap-basic1} \\
	&& \int_{t_1}^{t}\Vert\nabla\varphi_s (x(s))\Vert^2 ds\leq  \frac{2\delta}{\beta(\delta-\lambda)}E_{p}(t_1 )+\frac{\delta\Vert x^{*}\Vert^2 }{\beta(\delta-\lambda)}\int_{t_1}^{t}G(s)ds, \label{Lyap-basic11}
	\end{eqnarray}
where 
\begin{equation} \label{def:G} 
G(t)=(\lambda c+2a)\lambda\dfrac{\dot{\varepsilon}^{2}(t)}{\varepsilon^{\frac{3}{2}}(t)}-\dot{\varepsilon}(t)+ (1-p)\beta \lambda(\delta-\lambda)\varepsilon^{2}(t) .
	\end{equation}	
\end{theorem}
\begin{proof} Since the mapping $t \mapsto x_{\epsilon (t)}$ is absolutely continuous (indeed locally Lipschitz) the classical derivation  chain rule can be applied to compute the derivative of the function $E_p(\cdot)$, see \cite[section VIII.2]{Bre2}.
According to Lemma \ref{lem1} $i)$,  for almost all $t\geq t_0$
the derivative of  $E_p(\cdot)$  is given by:
	\begin{equation}\label{6}
	\begin{array}{lll}
	\dot{E_p}(t)&=&  \langle \nabla\varphi_{t}(x(t)),\dot{x}(t)\rangle+\dfrac{1}{2}\dot{\varepsilon}(t)\|x(t)\|^{2}-\dfrac{1}{2}\dot{\varepsilon}(t)\|x_{\varepsilon(t)}\|^{2}+ \langle\dot{v_p}(t),v_p(t)\rangle.
	\end{array} 
	\end{equation}
According to the definition of $v_p$ (see \eqref{3bbis}), and the  equation \eqref{1principal-bis}, the derivation of  $v_p$ gives
	$$
	\begin{array}{lll}
	\dot{v_p}(t)&=&   \dfrac{\lambda}{2}\dfrac{\dot{\varepsilon}(t)}{\sqrt{\varepsilon(t)}}\left(x(t)-x_{\varepsilon(t)}\right)+ \lambda\sqrt{\varepsilon(t)}\dot{x}(t)-\lambda\sqrt{\varepsilon(t)}\dfrac{d}{dt}x_{\varepsilon(t)}+\ddot{x}(t)\\
	&+&\beta\dfrac{d}{dt}\left(\nabla \varphi_{t}(x(t))+(p-1)\varepsilon(t)x(t)\right) \\ 
	& = & \dfrac{\lambda}{2}\dfrac{\dot{\varepsilon}(t)}{\sqrt{\varepsilon(t)}}\left(x(t)-x_{\varepsilon(t)}\right)+ \left(\lambda-\delta\right)\sqrt{\varepsilon(t)}\dot{x}(t)-\lambda\sqrt{\varepsilon(t)}\dfrac{d}{dt}x_{\varepsilon(t)}-\nabla\varphi_{t}(x(t)).
	\end{array}
	 $$
Let us write  shortly  $A_p (t) \eqdef \nabla \varphi_{t}(x(t))+(p-1)\varepsilon(t)x(t)$. We get
	\begin{equation}\label{7}
	\begin{array}{lll}
	\langle\dot{v}_p(t),v_p(t)\rangle&=&  \left\langle \dfrac{\lambda}{2}\dfrac{\dot{\varepsilon}(t)}{\sqrt{\varepsilon(t)}}\left(x(t)-x_{\varepsilon(t)}\right)+ \left(\lambda-\delta\right)\sqrt{\varepsilon(t)}\dot{x}(t)-\lambda\sqrt{\varepsilon(t)}\dfrac{d}{dt}x_{\varepsilon(t)}-\nabla\varphi_{t}(x(t)),v_p(t)\right\rangle   \\ 
	&=	&\dfrac{\lambda^{2}}{2}\dot{\varepsilon}(t) \|x(t)-x_{\varepsilon(t)}\|^{2}+\lambda\left(\dfrac{\dot{\varepsilon}(t)}{2\sqrt{\varepsilon(t)}}+\left(\lambda-\delta\right)\varepsilon(t)\right)\langle x(t)-x_{\varepsilon(t)},\dot{x}(t)\rangle\\
	& & +\underbrace{\left(\lambda-\delta\right)\sqrt{\varepsilon(t)}\|\dot{x}(t)\|^{2}+\beta\left(\lambda-\delta\right)\sqrt{\varepsilon(t)}\langle A_p (t),\dot{x}(t)\rangle}_{C_0}\\
	& & +\lambda\underbrace{\left(\dfrac{\beta\dot{\varepsilon}(t)}{2\sqrt{\varepsilon(t)}}-\sqrt{\varepsilon(t)}\right)\langle\nabla\varphi_{t}(x(t)),x(t)-x_{\varepsilon(t)})\rangle}_{=D_0} -\langle\nabla\varphi_{t}(x(t)),\dot{x}(t)\rangle,\\
	&&-\beta\langle\nabla \varphi_{t}(x(t)),A_p (t)\rangle-\beta \lambda\sqrt{\varepsilon(t)}\langle A_p (t),\dfrac{d}{dt}x_{\varepsilon(t)}\rangle-\lambda^{2}\varepsilon(t)\langle\dfrac{d}{dt} x_{\varepsilon(t)},x(t)-x_{\varepsilon(t)}\rangle\\
	&&-\lambda\sqrt{\varepsilon(t)}\langle\dfrac{d}{dt} x_{\varepsilon(t)},\dot{x}(t)\rangle+\dfrac{(p-1)\lambda\beta}{2}\dot{\varepsilon}(t)\sqrt{\varepsilon(t)}\langle x(t)-x_{\varepsilon(t)},x(t)\rangle .
	\end{array} 
	\end{equation}
	Since $\varphi_{t}$ is $\varepsilon (t)$-strongly convex, we have
	$$
	\varphi_{t}(x_{\varepsilon(t)})-\varphi_{t}(x(t))\geq \left\langle \nabla\varphi_{t}(x(t)),x_{\varepsilon(t)}-x(t) \right\rangle+\frac{\varepsilon(t)}{2} \| x(t)-x_{\varepsilon(t)}\|^{2}.
	$$
	Recall that $\varepsilon(t)$ is nonincreasing, \ie $\dot{\varepsilon}(t)\leq 0$. Therefore,  by  using the above estimation, we get  
	\begin{equation}\label{8}
	\begin{array}{lll}
	D_0	& \leq  & \left(\dfrac{\beta\dot{\varepsilon}(t)}{2\sqrt{\varepsilon(t)}}-\sqrt{\varepsilon(t)}\right)\left(\varphi_{t}(x(t))-\varphi_{t}(x_{\varepsilon(t)})\right)+\dfrac{1}{2}\left(\dfrac{\beta\dot{\varepsilon}(t)}{2\sqrt{\varepsilon(t)}}-\sqrt{\varepsilon(t)}\right)\varepsilon(t)\| x(t)-x_{\varepsilon(t)}\|^{2} .
	\end{array} 
	\end{equation}
For all $a>1$ we have the following elementary inequalities 
\begin{eqnarray}
	&& -\lambda\sqrt{\varepsilon(t)}\langle\dfrac{d}{dt} x_{\varepsilon(t)},\dot{x}(t)\rangle \leq\dfrac{\lambda\sqrt{\varepsilon(t)}}{2a}\|\dot{x}(t)\|^{2} +\dfrac{a\lambda\sqrt{\varepsilon(t)}}{2}\|\dfrac{d}{dt} x_{\varepsilon(t)}\|^{2}
	\label{parameter_a} \\
	&&- \beta \lambda\sqrt{\varepsilon(t)}\langle A_p (t),\dfrac{d}{dt}x_{\varepsilon(t)}\rangle\leq \dfrac{\lambda\beta^{2}\sqrt{\varepsilon(t)}}{2a}\|A_p (t)\|^{2} +\dfrac{ a\lambda\sqrt{\varepsilon(t)}}{2}\|\dfrac{d}{dt} x_{\varepsilon(t)}\|^{2}.\label{parameter_a2}
	\end{eqnarray}
Similarly, for all $b>0$,
\begin{equation}\label{parameter_b}
-\lambda^{2}\varepsilon(t)\langle\dfrac{d}{dt} x_{\varepsilon(t)},x(t)-x_{\varepsilon(t)}\rangle\leq \dfrac{b\lambda\sqrt{\varepsilon(t)}}{2}\|\dfrac{d}{dt} x_{\varepsilon(t)}\|^{2} +\dfrac{\lambda^{3}\varepsilon^{\frac{3}{2}}(t)}{2b}\|x(t)-x_{\varepsilon(t)}\|^{2}.
\end{equation}
The coefficients $a$ and $b$ will be adjusted later, conveniently. Note that 
 \begin{eqnarray}
C_0  &=&   \left(\lambda-\delta \right)\sqrt{\varepsilon(t)}\left(\|\dot{x}(t)\|^{2}+\beta\langle A_p (t),\dot{x}(t)\rangle \right) \nonumber \\ 
	&=&  \dfrac{\left(\lambda-\delta \right)\sqrt{\varepsilon(t)}}{2} \left(\|\dot{x}(t)\|^{2}+\|\dot{x}(t)+\beta A_p (t)\|^{2}-\beta^{2}\|A_p (t)\|^{2} \right). \label{c0}
	\\
-\beta\langle\nabla \varphi_{t}(x(t)),A_p (t)\rangle &=&- \dfrac{\beta}{2}\left[\|\nabla \varphi_{t}(x(t))\|^{2}+\|A_p (t)\|^{2}-\|\nabla \varphi_{t}(x(t))-A_p (t)\|^{2}\right] \nonumber \\
	&=&   -\dfrac{\beta}{2}\left[\|\nabla \varphi_{t}(x(t))\|^{2}
	+ \|A_p (t)\|^{2}-(p-1)^{2}\varepsilon^{2}(t)\|x(t)\|^{2}\right] .
	\label{c1}
	\end{eqnarray}
	By combining the inequalities \eqref{8}-\eqref{parameter_a}-	\eqref{parameter_a2}-\eqref{parameter_b}-\eqref{c0}-\eqref{c1} with \eqref{7}, we obtain
	\begin{equation}\label{7b}
	\begin{array}{lll}
	\langle\dot{v}_p(t),v_p(t)\rangle 
	&\leq&  \lambda\left(\dfrac{\beta\dot{\varepsilon}(t)}{2\sqrt{\varepsilon(t)}}-\sqrt{\varepsilon(t)}\right)\left(\varphi_{t}(x(t))-\varphi_{t}(x_{\varepsilon(t)})\right)\\
	&&+\lambda\left(\dfrac{\dot{\varepsilon}(t)}{2\sqrt{\varepsilon(t)}}+\left(\lambda-\delta\right)\varepsilon(t)\right)\langle x(t)-x_{\varepsilon(t)},\dot{x}(t)\rangle\\	
	&&+\left[\dfrac{\lambda^{2}}{2}\dot{\varepsilon}(t)+\dfrac{\lambda\varepsilon(t)}{2}\left(\dfrac{\beta\dot{\varepsilon}(t)}{2\sqrt{\varepsilon(t)}}-\sqrt{\varepsilon(t)}\right)+\dfrac{\lambda^{3}\varepsilon^{\frac{3}{2}}(t)}{2b}\right]\|x(t)-x_{\varepsilon(t)}\|^{2}\\
	&&+\dfrac{1}{2}\left(\left(1+\frac{1}{a}\right)\lambda-\delta\right)\sqrt{\varepsilon(t)}\|\dot{x}(t)\|^{2}+\dfrac{1}{2}\left(\lambda-\delta\right)\sqrt{\varepsilon(t)}\|\dot{x}(t)
	+\beta A_p (t)\|^{2}\\
	&&+\dfrac{1}{2}\left(2a+b\right)\lambda\sqrt{\varepsilon(t)}\|\dfrac{d}{dt}x_{\varepsilon(t)}\|^{2}+\dfrac{(p-1)\lambda\beta}{2}\dot{\varepsilon}(t)\sqrt{\varepsilon(t)}\langle x(t)-x_{\varepsilon(t)},x(t)\rangle\\
	&&+ \dfrac{\beta}{2}\left[\dfrac{\beta \lambda\sqrt{\varepsilon(t)}}{a}-1-\beta(\lambda-\delta)\sqrt{\varepsilon(t)})\right]\|A_p (t)\|^2-\langle\nabla\varphi_{t}(x(t)),\dot{x}(t)\rangle \\
	&&+\dfrac{\beta}{2}(p-1)^{2}\varepsilon^{2}(t)\|x(t)\|^{2}
-\dfrac{\beta}{2}\|\nabla \varphi_{t}(x(t))\|^{2} .
	\end{array} 
	\end{equation}
Combining  \eqref{7b} with \eqref{6}, the  terms 
$\langle\nabla\varphi_{t}(x(t)),\dot{x}(t)\rangle$ cancel each other out, which gives
	\begin{eqnarray}\label{9}
	\dot{E}_p(t)&\leq  & \lambda\left(\dfrac{\beta\dot{\varepsilon}(t)}{2\sqrt{\varepsilon(t)}}-\sqrt{\varepsilon(t)}\right)\left(\varphi_{t}(x(t))-\varphi_{t}(x_{\varepsilon(t)})\right)+\dfrac{1}{2}\left[\dot{\varepsilon}(t)+\beta(p-1)^{2}\varepsilon^{2}(t)
	\right]\|x(t)\|^{2}  \nonumber \\
	&& -\dfrac{1}{2}\dot{\varepsilon}(t)\|x_{\varepsilon(t)}\|^{2} +\lambda\left(\dfrac{\dot{\varepsilon}(t)}{2\sqrt{\varepsilon(t)}}+\left(\lambda-\delta\right)\varepsilon(t)\right)\langle x(t)-x_{\varepsilon(t)},\dot{x}(t)\rangle \nonumber\\
	&&+\left[\dfrac{\lambda^{2}}{2}\dot{\varepsilon}(t)+\dfrac{\lambda\varepsilon(t)}{2}\left(\dfrac{\beta\dot{\varepsilon}(t)}{2\sqrt{\varepsilon(t)}}-\sqrt{\varepsilon(t)}\right)+\dfrac{\lambda^{3}\varepsilon^{\frac{3}{2}}(t)}{2b}\right]\|x(t)-x_{\varepsilon(t)}\|^{2}\\
	&&+\dfrac{1}{2}\left(\left(1+\frac{1}{a}\right)\lambda-\delta\right)\sqrt{\varepsilon(t)}\|\dot{x}(t)\|^{2}+\dfrac{1}{2}\left(\lambda-\delta\right)\sqrt{\varepsilon(t)}\|\dot{x}(t)+\beta A_p (t)\|^{2} \nonumber \\
	&&+\dfrac{1}{2}\left(2a+b\right)\lambda\sqrt{\varepsilon(t)}\|\dfrac{d}{dt}x_{\varepsilon(t)}\|^{2}+\dfrac{(p-1)\lambda\beta}{2}\dot{\varepsilon}(t)\sqrt{\varepsilon(t)}\langle x(t)-x_{\varepsilon(t)},x(t)\rangle \nonumber \\
	&&+ \dfrac{\beta}{2}\left[\dfrac{\beta \lambda\sqrt{\varepsilon(t)}}{a}-1-\beta (\lambda-\delta)\sqrt{\varepsilon(t)})\right]\|A_p (t)\|^2-\dfrac{\beta}{2}\|\nabla \varphi_{t}(x(t))\|^{2} . \nonumber
	\end{eqnarray}
To build the differential inequality satisfied by $E_p (\cdot)$, let us majorize
	\begin{equation}\label{100}
	\begin{array}{lll}
	\mu(t)E_p(t)& = & \mu(t)\left(\varphi_{t}(x(t))-\varphi_{t}(x_{\varepsilon(t)})\right) +\dfrac{\mu(t)}{2}\|v_p(t)\|^{2} \vspace{1mm} \\ 
	& = & \mu(t)\left(\varphi_{t}(x(t))-\varphi_{t}(x_{\varepsilon(t)})\right) +\dfrac{\mu(t)\lambda^{2}\varepsilon(t)}{2}\|x(t)-x_{\varepsilon(t)}\|^{2}+\dfrac{\mu(t)}{2}\|\dot{x}(t)+\beta A_p (t)\|^{2}\\
	&&+\mu(t)\lambda\sqrt{\varepsilon(t)}\langle x(t)-x_{\varepsilon(t)}, \dot{x}(t)+\beta A_p  (t)\rangle\\
	&\leq&\mu(t)\left(\varphi_{t}(x(t))-\varphi_{t}(x_{\varepsilon(t)})\right) +\dfrac{\mu(t)\lambda^{2}\varepsilon(t)}{2}\|x(t)-x_{\varepsilon(t)}\|^{2}+\dfrac{\mu(t)}{2}\|\dot{x}(t)+\beta A_p (t)\|^{2}\\
	&&  +\mu(t)\lambda\sqrt{\varepsilon(t)}\langle x(t)-x_{\varepsilon(t)}, \dot{x}(t)\rangle+\beta\mu(t)\lambda\sqrt{\varepsilon(t)}\langle x(t)-x_{\varepsilon(t)}, \nabla \varphi_{t}(x(t)) \rangle\\
	&&
	+(p-1)\beta\mu(t)\lambda\varepsilon(t)\sqrt{\varepsilon(t)}\langle x(t)-x_{\varepsilon(t)}, x(t) \rangle \vspace{1mm}\\
	&\leq&\mu(t)\left(\varphi_{t}(x(t))-\varphi_{t}(x_{\varepsilon(t)})\right) +\mu(t)\lambda^{2}\varepsilon(t)\|x(t)-x_{\varepsilon(t)}\|^{2}+\dfrac{\mu(t)}{2}\|\dot{x}(t)+\beta A_p (t)\|^{2}\\
	&&  +\mu(t)\lambda\sqrt{\varepsilon(t)}\langle x(t)-x_{\varepsilon(t)}, \dot{x}(t)\rangle+\frac{\beta^{2}}{2}\mu(t)\|\nabla \varphi_{t}(x(t)) \|^{2}\\
	&&
	+(p-1)\beta\mu(t)\lambda\varepsilon(t)\sqrt{\varepsilon(t)}\langle x(t)-x_{\varepsilon(t)}, x(t) \rangle .
	\end{array} 
	\end{equation}
	By adding \eqref{9} and \eqref{100},  using $\mu(t)=-\dfrac{\dot{\varepsilon}(t)}{2\varepsilon (t)}+(\delta-\lambda )\sqrt{\varepsilon (t)}$, and after simplification, we get
	\begin{equation}\label{13a}
	\begin{array}{lll}
	\dot{E}_p(t)+\mu(t)E_p(t)&\leq &\sqrt{\varepsilon(t)}\left(-\dfrac{\dot{\varepsilon}(t)}{2\varepsilon^{\frac{3}{2}}(t)}+(\delta-2\lambda)+\lambda\beta\dfrac{\dot{\varepsilon}(t)}{2\varepsilon(t)}\right)\left(\varphi_{t}(x(t))-\varphi_{t}(x_{\varepsilon(t)})\right)\\
	&&+\dfrac{1}{2}\left[\dot{\varepsilon}(t)+ \beta(p-1)^{2}\varepsilon^{2}(t)
	\right]\|x(t)\|^{2}-\dfrac{1}{2}\dot{\varepsilon}(t)\|x_{\varepsilon(t)}\|^{2}\\ 
	&&\dfrac{\lambda}{4}\left[\beta\sqrt{\varepsilon(t)}\dot{\varepsilon}(t)+ 2\left(2\delta \lambda-2\lambda^{2}-1\right)\varepsilon^{\frac{3}{2}}(t)+2\frac{\lambda^{2}}{b}\varepsilon^{\frac{3}{2}}(t)\right]\|x(t)-x_{\varepsilon(t)}\|^{2}\\
	&&+\dfrac{1}{2}\left((1+\frac{1}{a})\lambda-\delta\right)\sqrt{\varepsilon(t)}\|\dot{x}(t)\|^{2}-\dfrac{\dot{\varepsilon}(t)}{4\varepsilon(t)}\|\dot{x}(t)+\beta A_p (t)\|^{2}\\
	&&+\dfrac{1}{2}\left(2a+b\right)\lambda \sqrt{\varepsilon(t)}\|\dfrac{d}{dt}x_{\varepsilon(t)}\|^{2}+(p-1)\beta \lambda(\delta-\lambda)\varepsilon^{2}(t)\langle x(t)-x_{\varepsilon(t)}, x(t) \rangle\\
	&&+ \dfrac{\beta}{2}\left[\dfrac{\beta \lambda\sqrt{\varepsilon(t)}}{a}-1-\beta (\lambda-\delta)\sqrt{\varepsilon(t)})\right]\|A_p (t)\|^2\\
	&&+\dfrac{\beta}{2}\left(-1-\beta\dfrac{\dot{\varepsilon}(t)}{2\varepsilon(t)}+\beta\left(\delta-\lambda \right)\sqrt{\varepsilon(t)}\right)\|\nabla \varphi_{t}(x(t))\|^{2} .
	\end{array} 
	\end{equation}
	Since $\varepsilon(\cdot)$ is nonincreasing, 
$	-\dfrac{\dot{\varepsilon}(t)}{4\varepsilon(t)}\|\dot{x}(t)+\beta A_p (t)\|^{2}\leq -\dfrac{\dot{\varepsilon}(t)}{2\varepsilon(t)}\|\dot{x}(t)\|^{2}-\dfrac{\beta^{2}\dot{\varepsilon}(t)}{2\varepsilon(t)}\|A_p (t)\|^{2}.$\\
Combining this inequality with 	
	$\langle x(t)-x_{\varepsilon(t)}, x(t) \rangle=\dfrac{1}{2}\left[\|x(t)-x_{\varepsilon(t)}\|^{2}+\|x(t)\|^{2}-\|x_{\varepsilon(t)}\|^{2}\right]$, 	 
	 we get
	\begin{equation}\label{13}
	\begin{array}{lll}
	\dot{E}_p(t)+\mu(t)E_p(t)&\leq &\sqrt{\varepsilon(t)}\left(-\dfrac{\dot{\varepsilon}(t)}{2\varepsilon^{\frac{3}{2}}(t)}+(\delta-2\lambda)+\lambda\beta\dfrac{\dot{\varepsilon}(t)}{2\varepsilon(t)}\right)\left(\varphi_{t}(x(t))-\varphi_{t}(x_{\varepsilon(t)})\right)\\
	&+&\dfrac{1}{2}\left[\dot{\varepsilon}(t)+\beta(p-1)\left( p-1+\lambda(\delta-\lambda)\right)\varepsilon^{2}(t)
	\right]\|x(t)\|^{2}\\
	&-&\dfrac{1}{2}\left[\dot{\varepsilon}(t)+(p-1)\beta \lambda(\delta-\lambda)\varepsilon^{2}(t)\right]\|x_{\varepsilon(t)}\|^{2}\\ 
	&+&\dfrac{\lambda}{4}\Big[\beta\sqrt{\varepsilon(t)}\dot{\varepsilon}(t)+2\left(2\delta \lambda-2\lambda^{2}-1\right)\varepsilon^{\frac{3}{2}}(t)+2\frac{\lambda^{2}}{b}\varepsilon^{\frac{3}{2}}(t)\\
	&+&2(p-1)\beta(\delta-\lambda)\varepsilon^{2}(t)\Big]\|x(t)-x_{\varepsilon(t)}\|^{2}\\
	&+&\left[\dfrac{1}{2}\left((1+\frac{1}{a})\lambda-\delta\right)\sqrt{\varepsilon(t)}-\dfrac{\dot{\varepsilon}(t)}{2\varepsilon(t)}\right]\|\dot{x}(t)\|^{2}+\dfrac{1}{2}\left(2a+b\right)\lambda\sqrt{\varepsilon(t)}\|\dfrac{d}{dt}x_{\varepsilon(t)}\|^{2}\\
	&+& \dfrac{\beta}{2}\left[\dfrac{\beta \lambda\sqrt{\varepsilon(t)}}{a}-1-\beta(\lambda-\delta)\sqrt{\varepsilon(t)})-\dfrac{\beta\dot{\varepsilon}(t)}{\varepsilon(t)}\right]\|A_p (t)\|^2\\
	&+&\dfrac{\beta}{2}\left(-1-\beta\dfrac{\dot{\varepsilon}(t)}{2\varepsilon(t)}+\beta\left(\delta-\lambda \right)\sqrt{\varepsilon(t)}\right)\|\nabla \varphi_{t}(x(t))\|^{2}.
	\end{array} 
	\end{equation}
	By using Lemma \ref{lem1}, we have 	
	$$
	\left\|\dfrac{d}{dt}x_{\varepsilon(t)}\right\|^{2}\leq \dfrac{\dot{\varepsilon}^{2}(t)}{\varepsilon^{2}(t)}\|x_{\varepsilon(t)}\|^{2}\leq \dfrac{\dot{\varepsilon}^{2}(t)}{\varepsilon^{2}(t)}\|x^{*}\|^{2},
	 $$
which  gives
\begin{equation}\label{14a}
	\begin{array}{lll}
	\dot{E}_p(t)+\mu(t)E_p(t)&\leq &\sqrt{\varepsilon(t)}\left(-\dfrac{\dot{\varepsilon}(t)}{2\varepsilon^{\frac{3}{2}}(t)}+(\delta-2\lambda)+\lambda\beta\dfrac{\dot{\varepsilon}(t)}{2\varepsilon(t)}\right)\left(\varphi_{t}(x(t))-\varphi_{t}(x_{\varepsilon(t)})\right)\\
	&&+\dfrac{1}{2}\left[\dot{\varepsilon}(t)+\beta(p-1)\left( p-1+\lambda(\delta-\lambda)\right)\varepsilon^{2}(t)
	\right]\|x(t)\|^{2}\\
	&&+\dfrac{\lambda}{4}\Big[\beta\sqrt{\varepsilon(t)}\dot{\varepsilon}(t)+ 2\left(2\delta \lambda-2\lambda^{2}-1\right)\varepsilon^{\frac{3}{2}}(t)+2\frac{\lambda^{2}}{b}\varepsilon^{\frac{3}{2}}(t)\\
	&&+2(p-1)\beta(\delta-\lambda)\varepsilon^{2}(t)\Big]\|x(t)-x_{\varepsilon(t)}\|^{2}\\
	&&+\left[\dfrac{1}{2}\left((1+\frac{1}{a})\lambda-\delta\right)\sqrt{\varepsilon(t)}-\dfrac{\dot{\varepsilon}(t)}{2\varepsilon(t)}\right]\|\dot{x}(t)\|^{2}\\
	&&+\dfrac{1}{2}\left[\left(2a+b\right)\lambda\dfrac{\dot{\varepsilon}^{2}(t)}{\varepsilon^{\frac{3}{2}}(t)}-\dot{\varepsilon}(t)+(1-p)\beta \lambda(\delta-\lambda)\varepsilon^{2}(t)\right]\|x_{\varepsilon(t)}\|^{2}\\
	&&+ \dfrac{\beta}{2}\left[\dfrac{\beta \lambda\sqrt{\varepsilon(t)}}{a}-1-\beta(\lambda-\delta)\sqrt{\varepsilon(t)}-\dfrac{\beta\dot{\varepsilon}(t)}{\varepsilon(t)}\right]\|A_p (t)\|^2\\
	&&+\dfrac{\beta}{2}\left(-1-\beta\dfrac{\dot{\varepsilon}(t)}{2\varepsilon(t)}+\beta\left(\delta-\lambda\right)\sqrt{\varepsilon(t)}\right)\|\nabla \varphi_{t}(x(t))\|^{2}.
	\end{array} 
	\end{equation}
Let us make precise the choice of the parameter $b$ and take
$$
b:= \frac{c}{2}\lambda   \mbox{ with }  c>2.
$$
Let us analyze the sign of the coefficients involved in \eqref{14a}. Since $\varepsilon(t)$ is nonincreasing and $p\leq1,$ 
	\begin{eqnarray}
	&&\dot{E}_p(t)+\mu(t)E_p(t)\;\;\leq \;\;\sqrt{\varepsilon(t)}\left(-\underbrace{\dfrac{\dot{\varepsilon}(t)}{2\varepsilon^{\frac{3}{2}}(t)}+(\delta-2\lambda)}_{=A}+\lambda\beta\underbrace{\dfrac{\dot{\varepsilon}(t)}{2\varepsilon(t)}}_{\leq 0}\right)\left(\varphi_{t}(x(t))-\varphi_{t}(x_{\varepsilon(t)})\right)
	\nonumber \\
	&&+\dfrac{1}{2}\left[\underbrace{\dot{\varepsilon}(t)+\beta(p-1)\left( p-1+\lambda(\delta-\lambda)\right)\varepsilon^{2}(t)}_{=B}
	\right]\|x(t)\|^{2}  \nonumber\\ 
	&&+\dfrac{\lambda}{4}\left[\underbrace{\beta\sqrt{\varepsilon(t)}\dot{\varepsilon}(t)}_{\leq 0}+2\left(\underbrace{2\left(\delta+\frac{1}{c}\right) \lambda-2\lambda^{2}-1}_{=C}\right)\varepsilon^{\frac{3}{2}}(t)+2\underbrace{(p-1)\beta(\delta-\lambda)}_{\leq 0 \, \text{since} \; \delta\geq \lambda}\varepsilon^{2}(t)\right]\|x(t)-x_{\varepsilon(t)}\|^{2}  \nonumber \\
	&&+\underbrace{\left[\dfrac{1}{2}\left((1+\frac{1}{a})\lambda-\delta\right)\sqrt{\varepsilon(t)}-\dfrac{\dot{\varepsilon}(t)}{2\varepsilon(t)}\right]}_{=D}\|\dot{x}(t)\|^{2}+ \dfrac{\beta}{2}\underbrace{\left[\beta\left(\dfrac{ \lambda}{a}+\delta-\lambda\right)\sqrt{\varepsilon(t)}-1-\dfrac{\beta\dot{\varepsilon}(t)}{\varepsilon(t)}\right]}_{\Delta_1}\|A_p (t)\|^2  \nonumber \\
	&&+\dfrac{1}{2}\left[\left(2a+c\lambda\right)\lambda\dfrac{\dot{\varepsilon}^{2}(t)}{\varepsilon^{\frac{3}{2}}(t)}-\dot{\varepsilon}(t)+(1-p)\beta \lambda(\delta-\lambda)\varepsilon^{2}(t)\right]\|x_{\varepsilon(t)}\|^{2} \nonumber \\
	&&+\dfrac{\beta}{2}\underbrace{\left(-1-\beta\dfrac{\dot{\varepsilon}(t)}{2\varepsilon(t)}+\beta\left(\delta-\lambda\right) \sqrt{\varepsilon(t)}\right)}_{\Delta_2}\|\nabla \varphi_{t}(x(t))\|^{2} . \label{14}
	\end{eqnarray}
$\bullet$ \; Condition $(\mathcal{H}_p)$ implies that 
	$$\frac{d}{dt}\left(\dfrac{1}{\sqrt{\varepsilon(t)}}\right) \leq (2\lambda-\delta),\quad
	\frac{d}{dt}\left(\dfrac{1}{\sqrt{\varepsilon(t)}}\right) \leq \frac{1}{2}\left(\delta-(1+\frac{1}{a})\lambda\right) \quad \text{and}\quad \delta\beta\leq\dfrac{1}{\sqrt{\varepsilon(t)}}.$$
	So, we have 
$$A=-\dfrac{\dot{\varepsilon}(t)}{2\varepsilon^{\frac{3}{2}}(t)}+(\delta-2\lambda)=\dfrac{d}{dt}\left(\dfrac{1}{\sqrt{\varepsilon(t)}}\right)+(\delta-2\lambda)\leq 0.$$
\smallskip
$\bullet$ \; 	According to  $\delta>2\sqrt{(1-p)}$ and $ \lambda < \frac{\delta+\sqrt{\delta^{2}-4(1-p)}}{2},$ we have   $p-1+\delta\lambda-\lambda^{2}\geq 0$.\\
 Therefore, 
$$
B	= \underbrace{\dot{\varepsilon}(t)}_{\leq 0}+\beta(p-1)\left( p-1+\lambda(\delta-\lambda)\right)\varepsilon^{2}(t)  \leq   \beta\underbrace{(p-1)}_{\leq 0} \left( p-1+\delta\lambda-\lambda^{2}\right)\varepsilon^{2}(t)\leq 0.
$$	
\smallskip
$\bullet$ \;  
 When $\delta \leq \sqrt{2}-\frac{1}{c}$ we have
		$$
	2\left(\delta+\frac{1}{c}\right) \lambda-2\lambda^{2}-1 \leq 2\sqrt{2}\lambda-2\lambda^{2} -1 = -(\sqrt{2}\lambda-1)^2\leq 0.
		$$ 
 When $\delta > \sqrt{2}-\frac{1}{c},$ we have
	 $$(\delta+\frac{1}{c})\lambda-\lambda^{2}- \demi\leq 0, \mbox{ because } \lambda \geq \demi \left( \delta+\frac{1}{c}+\sqrt{(\delta+\frac{1}{c})^{2}-2} \right).
	 $$	
		 Therefore 
		$
		(\delta+\frac{1}{c}) \lambda-\lambda^{2}-\demi\leq 0,
		$
		which implies that 
		$$
		C=2\left(\delta+\frac{1}{c}\right) \lambda-2\lambda^{2}-1\leq 0.
		$$
\smallskip	
$\bullet$ \; Condition $(\mathcal{H}_p)$ implies that 
 $$D=\left[\dfrac{1}{2}\left((1+\frac{1}{a})\lambda-\delta\right)-\frac{1}{2}\dfrac{\dot{\varepsilon}(t)}{\varepsilon^{\frac{3}{2}}(t)}\right] =
 \left[\frac{d}{dt}\left(\dfrac{1}{\sqrt{\varepsilon(t)}}\right)+\frac{1}{2}\left((1+\frac{1}{a})\lambda-\delta\right)\right]\leq 0.
 $$
\smallskip
$\bullet$ \;  For nonpositivity of $\Delta_1$, we use $\delta\beta\leq \dfrac{1}{\sqrt{\varepsilon(t)}}$ and $\dfrac{d}{dt}\left(\dfrac{1}{\sqrt{\varepsilon}(t)}\right)\leq \dfrac{1}{2}\left(\delta-(1+\dfrac{1}{a})\lambda\right)$ to conclude
		\begin{equation*}
		\begin{array}{lll}
		\Delta_1	&  =&\beta\left(\dfrac{ \lambda}{a}+\delta-\lambda\right)\sqrt{\varepsilon(t)}-1-\dfrac{\beta\dot{\varepsilon}(t)}{\varepsilon(t)} \\
		& = & \beta\left(\dfrac{ \lambda}{a}+\delta-\lambda\right)\sqrt{\varepsilon(t)}-1+ 2\beta\sqrt{\varepsilon(t)}\underbrace{\left(-\dfrac{\dot{\varepsilon}(t)}{2\varepsilon^{\frac{3}{2}}(t)}-\frac{1}{2}\left(\delta-(1+\frac{1}{a})\lambda\right)\right)}_{\leq 0}+\beta\left(\delta-(1+\frac{1}{a})\lambda\right)\sqrt{\varepsilon(t)}\\
		&\leq& 2\beta\left(\delta-\lambda\right)\sqrt{\varepsilon(t)}-1=\beta\underbrace{\left(\delta-2\lambda\right)}_{\leq 0}\sqrt{\varepsilon(t)}+\underbrace{\beta\delta\sqrt{\varepsilon(t)}-1}_{\leq 0}\leq0.
\end{array}	
\end{equation*}
$\bullet$\; Finally,
\begin{equation*}\label{key}
		\begin{array}{lll}
		\Delta_2&=&\lambda\beta\sqrt{\varepsilon(t)}-1+\beta\sqrt{\varepsilon(t)}\underbrace{\left((\delta-2\lambda)+\frac{d}{dt}\left(\dfrac{1}{\sqrt{\varepsilon(t)}}\right)\right)}_{\leq 0}\\
       	&\leq &\lambda\left(\beta\sqrt{\varepsilon(t)}-\dfrac{1}{\delta}\right)+\dfrac{\lambda}{\delta}-1\leq \dfrac{\lambda}{\delta}-1 .
 \end{array}
 \end{equation*}
It follows from the estimate  \eqref{14} that
	\begin{equation*}
	\begin{array}{lll}
	\dot{E}_p(t)+\mu(t)E_p(t)&\leq &
	\dfrac{1}{2}\left[\left(2a+c\lambda\right)\lambda\dfrac{\dot{\varepsilon}^{2}(t)}{\varepsilon^{\frac{3}{2}}(t)}-\dot{\varepsilon}(t)+(1-p)\beta \lambda(\delta-\lambda)\varepsilon^{2}(t)\right]\|x_{\varepsilon(t)}\|^{2}\\
	&&+ \frac{\beta}{2}\left(\frac{\lambda}{\delta}-1\right)\|\nabla \varphi_{t}(x(t))\|^{2}.
	\end{array} 
	\end{equation*}
	Since $\|x_{\varepsilon(t)}\|\leq \|x^{*}\|,$  we get
	\begin{equation}\label{est14}
	\begin{array}{lll}
	\dot{E}_p(t)+\mu(t)E_p(t)&\leq &
	\dfrac{\|x^{*}\|^{2}}{2}G(t)+  \frac{\beta}{2}\left(\frac{\lambda}{\delta}-1\right)\|\nabla \varphi_{t}(x(t))\|^{2},
	\end{array} 
	\end{equation}	
	where 
	$
	G(t)= (\lambda c+2a)\lambda\dfrac{\dot{\varepsilon}^{2}(t)}{\varepsilon^{\frac{3}{2}}(t)}-\dot{\varepsilon}(t)+(1-p)\beta \lambda(\delta-\lambda)\varepsilon^{2}(t).$\\
We have $\frac{\lambda}{\delta}-1\leq0$. By taking $\gamma(t)=\exp\left(\displaystyle \int_{t_1}^{t} \mu(s)ds\right),$ and setting
\begin{equation}\label{def:Wp}
W_p(t)\eqdef e^{\int^{t}_{t_1}\mu(s)ds}E_p(t) 
\end{equation}
 we conclude that
	\begin{equation}\label{15}
	\dot{W}_p(t)=\gamma(t)\left(\dot{E}_p(t)+\mu(t)E_p(t)\right)\leq  \frac{\|x^{*}\|^{2}}{2}G(t)\gamma(t).
	\end{equation}
	By integrating \eqref{15} on $[t_1 , t]$, and dividing by $\gamma(t),$ we obtain  our claim \eqref{Lyap-basic1} 
	\begin{equation}
	E_p(t)\leq \dfrac{\|x^{*}\|^{2}}{2\gamma(t)}  \displaystyle\int_{t_1}^{t}G(s) \gamma(s) ds+ \dfrac{\gamma(t_1)E_{p}(t_1)}{\gamma(t)} .
	\end{equation}
	Coming back to (\ref{est14}), we get by integration 
		$$E_p (t)-E_p (t_1 )+\int_{t_1}^{t}\mu(s)E_p (s)ds+\frac{\beta}{2}\left(1-\frac{\lambda}{\delta}\right) \int_{t_1}^{t}\Vert\nabla\varphi_s(x(s))\Vert^2 ds\leq \frac{\Vert x^* \Vert^2}{2}\int_{t_1}^{t} G(s)ds.$$
Our assertion \eqref{Lyap-basic11}  is then reached by neglecting the positive term $E_p (t)+\int_{t_1}^{t}\mu(s)E_p (s)ds$. \qed
		\end{proof}
	\begin{corollary}\label{corol1}
		Let  $x(\cdot): [t_0, + \infty[ \to \cH$ be a solution trajectory of the system {\rm(TRISHE)} with $\delta>0$
		\begin{equation}\label{equp1}
		\ddot{x}(t)+\delta\sqrt{\varepsilon(t)}\dot{x}(t)+\beta\dfrac{d}{dt}\left(\nabla \varphi_{t}(x(t)))\right)+\nabla\varphi_{t}(x(t))=0.
		\end{equation}
		Let us assume that there exists $a,c>1$  and $t_1 \geq t_0$ such that for all $t \geq t_1 ,$  $( \mathcal{H}_1)$ holds.
		Then,  
		\begin{equation}\label{Lyap-basic20}
		E_1(t)\leq \dfrac{\|x^{*}\|^{2}}{2}\dfrac{\displaystyle \int_{t_1}^{t}\left[\left((\lambda c+2a)\lambda\dfrac{\dot{\varepsilon}^{2}(s)}{\varepsilon^{\frac{3}{2}}(s)}-\dot{\varepsilon}(s)\right) \gamma(s) \right]ds}{\gamma(t)}+ \dfrac{\gamma(t_1)E_1(t_1)}{\gamma(t)}
		\end{equation}
		where $\gamma(t)=\exp\left(\displaystyle \int_{t_1}^{t} \mu(s)ds\right)$ and $ \mu(t)=-\dfrac{\dot{\varepsilon}(t)}{2\varepsilon(t)}+(\delta-\lambda)\sqrt{\varepsilon(t)}$.
	\end{corollary}
	\begin{corollary}\label{corol2}
		Let  $x(\cdot): [t_0, + \infty[ \to \cH$ be a solution trajectory of the system {\rm(TRISH)} with $\delta>2$
		\begin{equation}\label{equp2}
		\ddot{x}(t)+\delta\sqrt{\varepsilon(t)}\dot{x}(t)+\beta\nabla^{2} f(x(t))\dot{x}(t)+\nabla\varphi_{t}(x(t))=0.
		\end{equation}
		Let us assume that there exists $a,c>1$  and $t_1 \geq t_0$ such that for all $t \geq t_1 ,$  condition $( \mathcal{H}_0)$ holds. Then, 
		\begin{equation}\label{Lyap-basic3}
		E_0(t)\leq \dfrac{\|x^{*}\|^{2}}{2}\dfrac{\displaystyle \int_{t_1}^{t}\left[\left((\lambda c+2a)\lambda\dfrac{\dot{\varepsilon}^{2}(s)}{\varepsilon^{\frac{3}{2}}(s)}-\dot{\varepsilon}(s)+\beta \lambda(\delta-\lambda)\varepsilon^{2}(t)\right) \gamma(s) \right]ds}{\gamma(t)}+ \dfrac{\gamma(t_1)E_0(t_1)}{\gamma(t)}
		\end{equation}
		where $\gamma(t)=\exp\left(\displaystyle \int_{t_1}^{t} \mu(s)ds\right)$ and $ \mu(t)=-\dfrac{\dot{\varepsilon}(t)}{2\varepsilon(t)}+(\delta-\lambda)\sqrt{\varepsilon(t)}$.
	\end{corollary}

\if
{
The following theorem gives an integral estimation of the gradients.	
It makes use of the energy function $W_p$ which has been defined in \eqref{def:Wp}.
\if
{
\begin{equation}\label{W}
W_p(t)=e^{\int^{t}_{t_0}\mu(s)ds}E_p(t)\;\text{ where }\mu(t)=-\dfrac{\dot{\varepsilon}(t)}{2\varepsilon(t)}+(\delta-K)\sqrt{\varepsilon(t)}.
\end{equation} 
which satisfies 
	\begin{equation}\label{W1a}
	\dfrac{d}{dt}W_p(t)=e^{\int^{t}_{t_0}\mu(s)ds}\left[\dfrac{d}{dt}E_p(t)+\mu(tE_p(t))\right] .
	\end{equation}
}
\fi

\begin{theorem}
Under condition of Theorem \ref{strong-conv-thm-b}, if  $\int_{t_0}^{+\infty}e^{\int_{t_0}^{s}\mu(\tau)d\tau}G(s)ds <+\infty$   then 
\begin{equation}\label{key10}
\displaystyle{\int_{t_1}^{+\infty}e^{\int_{t_0}^{s}\mu(\tau)d\tau}\Vert\nabla\varphi_s (x(s))\Vert^2 ds<+\infty}.
\end{equation}
If further, $\sup_{t\geq t_0}\left(e^{\int_{t_0}^{t}\mu(\tau)d\tau}\varepsilon(t)\right)<+\infty,$ then \;
$
f(x(t))-\min_{\cH} f= \mathcal O \left( \displaystyle\dfrac{1}{e^{\int_{t_0}^{t}\mu(\tau)d\tau} }   \right) .
$
\end{theorem}
\begin{proof}
Going back to the inequality \eqref{est14} of the proof of Theorem \ref{strong-conv-thm-b}, we have
\begin{eqnarray}
\dot{W}_p(t)+\gamma(t)\left(1-\frac{\lambda}{\delta}\right)\|\nabla \varphi_{t}(x(t))\|^{2}&=&\gamma(t)\Big(\dot{E}_p(t)+\mu(t)E_p(t)+\left(1-\frac{\lambda}{\delta}\right)\|\nabla \varphi_{t}(x(t))\|^{2} \Big) \nonumber
\\
&\leq&  \frac{\|x^{*}\|^{2}}{2}G(t)\gamma(t), \label{est_grad_10}
\end{eqnarray}
	where $G(t)=\left((\lambda c+2a)\lambda\dfrac{\dot{\varepsilon}^{2}(t)}{\varepsilon^{\frac{3}{2}}(t)}-\dot{\varepsilon}(t)+(1-p)\beta \lambda(\delta-\lambda)\varepsilon^{2}(t)\right)$ and $\gamma(t)=\exp\left(\displaystyle \int_{t_1}^{t} \mu(s)ds\right).$\\  
By integrating \eqref{est_grad_10} on $[t_1,t],$ we get 
\begin{equation}\label{}
\left(1-\frac{\lambda}{\delta}\right)\int_{t_1}^{t}\gamma(s)\|\nabla \varphi_{s}(x(s))\|^{2}ds + W_p(t)\leq W_p(t_1)+ 
\frac{\|x^{*}\|^{2}}{2} \int_{t_1}^{t} G(s)\gamma(s)ds .
\end{equation}
Since $\int_{t_0}^{+\infty}\gamma(s)G(s)ds <+\infty$, we infer  that
$\int_{t_0}^{+\infty}\gamma(s)\|\nabla \varphi_{s}(x(s))\|^{2}ds<+\infty$ and for $t\geq t_0,$ $W_p(t)$ is bounded. From
$$\gamma(t)\left(f(x(t))-\min_{\mathcal H}f\right)	
\leq  W_p(t)+\dfrac{\gamma(t)\varepsilon(t)}{2}\|x^*\|^{2}$$
and  since  $\gamma(t)\varepsilon(t)$ is bounded for $t\leq t_0,$ we conclude that 
$f(x(t))-\min_{\mathcal H}f=\mathcal O \left( \displaystyle\dfrac{1}{\gamma(t) }   \right).$  \qed
 \end{proof}
 
 }
 \fi

	\section{Particular cases}\label{sec:particular-cases}
Take $\e(t)=\displaystyle\frac{1}{t^{r} } $,  $0<r<2$, $t_0>0$, and consider the systems (TRISHE) and (TRISH). The convergence rate of the values and the strong convergence to the minimum norm solution will be obtained by particularizing  Theorem \ref{strong-conv-thm-b} to these situations. In these cases, the integrals that enter into the formulation of Theorem \ref{strong-conv-thm-b} can be calculated explicitly. 
In fact, obtaining sharp convergence rate of the gradients requires another Lyapunov analysis based on the function
 $\mathcal{E}_p$ defined by 
\begin{equation}\label{def:Ep}
\mathcal{E}_p(t)\eqdef 
	\left(\varphi_{t}(x(t))-\varphi_{t}(x_{\varepsilon(t)})\right) +\dfrac{1}{2}\|\dot{x}(t)+\beta\left[\nabla \varphi_{t}(x(t))+(p-1)\varepsilon(t)x(t)\right]\|^{2}.
	\end{equation}
We can notice that when $\lambda=0,$ we have $\mathcal{E}_p(t)=E_p(t)$. So, with $\lambda=0,$ the estimation \eqref{9} becomes
\begin{eqnarray}
\dot{\mathcal{E}}_p(t)=\dot{E}_p(t)&\leq  & \dfrac{1}{2}\left[\dot{\varepsilon}(t)+\beta(p-1)^{2}\varepsilon^{2}(t)
\right]\|x(t)\|^{2} -\dfrac{1}{2}\dot{\varepsilon}(t)\|x_{\varepsilon(t)}\|^{2}-\dfrac{\delta}{2}\sqrt{\varepsilon(t)}\|\dot{x}(t)\|^{2}  \nonumber\\
&&-\dfrac{\delta}{2}\sqrt{\varepsilon(t)}\|\dot{x}(t)+\beta A_p (t)\|^{2}+ \dfrac{\beta}{2}\left[-1+\beta \delta\sqrt{\varepsilon(t)})\right]\|A_p (t)\|^2-\dfrac{\beta}{2}\|\nabla \varphi_{t}(x(t))\|^{2}. \nonumber
\end{eqnarray}
By supposing $\delta\beta\leq\dfrac{1}{\sqrt{\varepsilon(t)}},$ we conclude that
\begin{eqnarray}\label{key-bbb}
\dot{\mathcal{E}}_p(t)&\leq  & \dfrac{1}{2}\left[\dot{\varepsilon}(t)+\beta(p-1)^{2}\varepsilon^{2}(t)
\right]\|x(t)\|^{2} -\dfrac{1}{2}\dot{\varepsilon}(t)\|x^{*}\|^{2}-\dfrac{\delta}{2}\sqrt{\varepsilon(t)}\|\dot{x}(t)\|^{2}  \nonumber\\
&&-\dfrac{\delta}{2}\sqrt{\varepsilon(t)}\|\dot{x}(t)+\beta A_p (t)\|^{2}- \dfrac{\beta}{2}\|\nabla \varphi_{t}(x(t))\|^{2}.
\end{eqnarray}

\subsection{System (TRISHE)}	
	\begin{theorem}\label{thm:model-a}
	Take $\e(t)=\displaystyle\frac{1}{t^{r} } $ and  $0<r<2$.
	Let $x : [t_0, +\infty[ \to \mathcal{H}$ be a solution trajectory of
	\begin{equation}\label{eqr1}
	\ddot{x}(t) +\left( \frac{\d}{ \displaystyle{t^{\frac{r}{2}}}}+\frac{\beta}{t^{r}}\right)\dot{x}(t) +\beta\nabla^{2} f\left(x(t) \right)\dot{x}(t)+ \nabla f\left(x(t) \right)+ \left(\frac{1}{t^r}-\frac{r\beta}{t^{r+1}}\right) x(t)=0.
	\end{equation}		
	Then, we have  convergence of  values,  strong convergence to the minimum norm solution, and
	\begin{eqnarray}
	&& 
	f(x(t))-\min_{\cH} f= \mathcal O \left( \displaystyle\frac{1}{t^{r} }   \right) \mbox{ as } \; t \to +\infty;\\
	&& \|x(t) -x_{\varepsilon(t)}\|^2=\mathcal{O}\left(\dfrac{1}{ t^{\frac{2-r}2}}\right) \mbox{ as } \; t \to +\infty.\\
	&&
	\|\dot{x}(t)+\beta \nabla f (x(t))\|=\mathcal{O}\left(\dfrac{1}{ t^{\min(\left(\frac{2+r}4,r\right)}}\right) 
	\mbox{ as } \; t \to +\infty.
	\end{eqnarray}
In addition, we have the following integral estimates 
	$$\int_{t_1}^{+\infty}t^{r-1}\Vert\dot{x}(t)\Vert^2 dt<+\infty, \quad \int_{t_1}^{+\infty}t^{(\frac{3r}{2}-1)}\Vert\nabla f(x(t))\Vert^2 dt<+\infty.$$
\end{theorem}
\begin{proof}
a) By taking $\e(t)=\displaystyle\frac{1}{t^{r} } $ in Corollary \ref{corol1}, we get \eqref{eqr1}. So if  $( \mathcal{H}_1)$ is satisfied,
 we  get 
	\begin{equation}\label{Lyap-basic2b}
	E_1(t)\leq \dfrac{\|x^{*}\|^{2}}{2\gamma(t)}\displaystyle \int_{t_1}^{t}\left( (\lambda c+2a)\lambda\dfrac{\dot{\varepsilon}^{2}(s)}{\varepsilon^{\frac{3}{2}}(s)}-\dot{\varepsilon}(s)\right) \gamma(s) ds+ \dfrac{\gamma(t_1)E_1(t_1)}{\gamma(t)}.
	\end{equation}
Using the same technique as before, we start by choosing the parameters  $a>1, c> 2$, 
 $\lambda>0$ such that
 \begin{center}
 	$\frac{\delta}{2}< \lambda <\frac{a}{a+1}\delta$ for $ 0 <\delta \leq \sqrt{2}-\frac{1}{c}$ and $\frac{\delta}{2} < \frac12\left(\delta+\frac{1}{c}+\sqrt{(\delta+\frac{1}{c})^{2}-2}\right) < \lambda <\frac{a}{a+1}\delta $ for $\delta > \sqrt{2}-\frac{1}{c}$.
 \end{center}
We can easily check that for $r<2$ and for  $t\geq t_1$ large enough, 
$$
\dfrac{d}{dt}\left(\dfrac{1}{\sqrt{\varepsilon(t)}}\right) =\frac{r}2t^{\frac{r-2}{2}}
\leq \min\left(2\lambda-\delta\; , \; \frac{1}{2} \left(\delta-\frac{a+1}{a}\lambda \right)\right),
\quad \mbox{and}\quad \beta \delta \leq t^{\frac{r}{2}}.$$
This expresses that the condition $( \mathcal{H}_1)$ holds. With the notations of  Theorem \ref{strong-conv-thm-b}, we have
	\begin{eqnarray}
	\mu(t)&=&-\dfrac{\dot{\varepsilon}(t)}{2\varepsilon(t)}+(\delta-\lambda)\sqrt{\varepsilon(t)}=\dfrac{r}{2t}+ \dfrac{\delta -\lambda}{t^{\frac{r}2}} \label{mu}\\
	\gamma(t)&=&\exp\left(\displaystyle \int_{t_1}^{t} \mu(s)ds\right) =  \left(\dfrac{t}{t_1}\right)^{\frac r2} \exp\left[\frac{2(\delta-\lambda)}{2-r}\left(t^{\frac{2-r}2} - t_1^{\frac{2-r}2} \right)\right] \nonumber\\
	& =& 
	C_1t^{\frac r2} \exp\left[\frac{2(\delta-\lambda)}{2-r}t^{\frac{2-r}2} \right]\quad \mbox{where} \:C_1=\left(t_1^{\frac r2} \exp\left[\frac{2(\delta-\lambda)}{2-r}t_1^{\frac{2-r}2} \right]\right)^{-1}. \label{gamma-est-1}
	\end{eqnarray}
Setting  $\lambda_0:=(\lambda c+2a)\lambda , \; \delta_0:=\dfrac{2(\delta-\lambda)}{2-r},$ and  replacing $\varepsilon (t)$ and $\gamma(t)$ by their values in
	\eqref{Lyap-basic2b}, we get
	\begin{equation}\label{key1000}
	\begin{array}{lll}
E_{1}(t)	&  \leq& \dfrac{r \|x^{*}\|^{2}}{2 t^{\frac r2} \exp\left(\delta_0t^{\frac{2-r}2} \right)}\displaystyle     \int_{t_1}^{t}\left( \dfrac{\lambda_0r}{s^2 } + \frac{1}{s^{\frac{r+2}2}}\right)\exp\left(\delta_0s^{\frac{2-r}2}\right)ds+\dfrac{\gamma(t_1)E_1(t_1)}{\gamma(t)} \\ 
	\end{array}
	\end{equation}	
Then notice that
$$
\dfrac{d}{ds}\left( \dfrac{1}{\rho s  } \exp\left(\delta_0 s^{\frac{2-r}2}\right)\right)= \left(  -\dfrac{1}{\rho s^2 } + \dfrac{\delta_0(2-r)}{2\rho s^{\frac{r+2}2}}   
\right)\exp\left(\delta_0s^{\frac{2-r}2}\right).
$$
For $s$ large enough, by taking $0<\rho < \frac1{a+1}\delta$, we  have $\dfrac{\lambda_0r}{s^2 } + \dfrac{1}{s^{\frac{r+2}2}} \leq -\dfrac{1}{\rho s^2 } + \dfrac{\delta_0(2-r)}{2\rho s^{\frac{r+2}2}}, $ which gives
\begin{equation*}
\begin{array}{lll}
E_{1}(t)	&  \leq& \dfrac{r}{2 t^{\frac r2} \exp\left(\delta_0t^{\frac{2-r}2} \right)}\displaystyle \int_{t_1}^{t}\left(  -\dfrac{1}{\rho s^2 } + \dfrac{\delta_0(2-r)}{2\rho s^{\frac{r+2}2}}  	\right)\exp\left(\delta_0s^{\frac{2-r}2}\right)ds+\dfrac{\gamma(t_1)E_1(t_1)}{\gamma(t)}\\
&=& \dfrac{1}{2  t^{\frac r2} \exp\left(\delta_0t^{\frac{2-r}2} \right)}\displaystyle\int_{t_1}^{t}\dfrac{d}{ds}\left( \dfrac{1}{\rho s } \exp\left(\delta_0s^{\frac{2-r}2}\right)\right)ds+\dfrac{\gamma(t_1)E_1(t_1)}{\gamma(t)}\\
&=& \dfrac{r}{ 2\rho t^{\frac{r+2}2}} - \dfrac{r}{ t^{\frac{r}2}\exp\left(\delta_0t^{\frac{2-r}2} \right)}  \dfrac{1}{2\rho t_1 } \exp\left(\delta_0t_1^{\frac{2-r}2}\right)+\dfrac{\gamma(t_1)E_1(t_1)}{\gamma(t)}
\leq \dfrac{r}{2\rho t^{\frac{r+2}2}}+\dfrac{\gamma(t_1)E_1(t_1)}{\gamma(t)}.
\end{array}
\end{equation*}
We have $\dfrac{\gamma(t_1)E_1(t_1)}{\gamma(t)} \leq C t^{-\frac r2} \exp\left[-\delta_0t^{\frac{2-r}2} \right]$.  Since $0 <r<2$ and $\delta_0 > 0$, we deduce that  $\dfrac{\gamma(t_1)E_1(t_1)}{\gamma(t)}$ tends to zero at an exponential rate, as $t \to +\infty.$ Therefore, there exists a positive constant $C$ such that for $t$ large enough
\begin{equation}\label{key10000}
E_1(t) \leq \dfrac{C}{ t^{\frac{r+2}{2}}}.
\end{equation}
By  Lemma \ref{lem-basic-b}, we deduce that  there exists positive constants $C$ and $M$ such that, for $t$ large enough,
$$
f(x(t))-\min_{\mathcal H}f \leq C \left(\dfrac{1}{ t^{\frac{r+2}{2}}}  + \frac{1}{t^r} \right),\quad \|x(t) - x_{\varepsilon(t)}\|^2  \leq \frac{2E_{p}(t)}{\varepsilon(t)}\leq \frac{2C}{t^{\frac{2-r}2}}, \; \mbox{ and } $$
\begin{equation}\label{vitess}
\|\dot{x}(t)+\beta\nabla \varphi_{t}(x(t))+(p-1)\varepsilon(t)x(t)\|^2  \leq M E_p(t)\leq \frac{MC}{t^{\frac{2+r}2}}
\end{equation}
Since $0<r<2,$   we conclude  that
$$
f(x(t))-\min_{\cH} f=\mathcal O \left( \displaystyle{ \frac{1}{t^{r}} }   \right) ,\quad 
\|x(t)-x_{\varepsilon(t)}\|^{2}=\mathcal{O}\left(\dfrac{1}{ t^{\frac{2-r}2}}\right),\mbox{ as } \; t \to +\infty.
$$
By \eqref{vitess}, we have 
$$\|\dot{x}(t)+\beta \nabla f (x(t))\|\leq \|\dot{x}(t)+\beta \nabla \varphi_{t}(x(t))\|+\varepsilon(t)\|x(t)\|\leq \frac{\sqrt{MC}}{t^{\frac{2+r}4}}+\frac{1}{t^{r}}\|x(t)\|.$$
Since $x$ is bounded, 
we conclude that, as $ \; t \to +\infty$,
$$
\|\dot{x}(t)+\beta \nabla f (x(t))\|=
\left\{\begin{array}{lcl}
\mathcal{O}\left(\dfrac{1}{ t^{\frac{2+r}4}}\right) & \mbox{ if } & r\in[\frac{2}{3},2[\\
\mathcal{O}\left(\dfrac{1}{ t^{r}}\right) & \mbox{ if } & r\in]0,\frac{2}{3}[.
\end{array}\right.
$$
b) We now come to the integral estimates of the velocities and gradient terms. For this, we use the
pointwise estimates already established, and proceed with the Lyapunov function $\mathcal{E}_p$ defined in \eqref{def:Ep}.	
The system (TRISHE) corresponds to $p=1$, so we consider
$$\mathcal{E}_1(t)\eqdef 
\left(\varphi_{t}(x(t))-\varphi_{t}(x_{\varepsilon(t)})\right) +\dfrac{1}{2}\|\dot{x}(t)+\beta\nabla \varphi_{t}(x(t))\|^{2}.$$
Since for  $t>t_1,$ $\delta\beta\leq\dfrac{1}{\sqrt{\varepsilon(t)}},$ then according to \eqref{key-bbb}, we have 
	\begin{eqnarray}\label{keyabc}
	\dot{\mathcal{E}}_1(t)&\leq  & \dfrac{1}{2}\dot{\varepsilon}(t)\|x(t)\|^{2} -\dfrac{1}{2}\dot{\varepsilon}(t)\|x^{*}\|^{2}-\dfrac{\delta}{2}\sqrt{\varepsilon(t)}\|\dot{x}(t)\|^{2}  \nonumber\\
	&&-\dfrac{\delta}{2}\sqrt{\varepsilon(t)}\|\dot{x}(t)+\beta \nabla \varphi_{t}(x(t))\|^{2}- \dfrac{\beta}{2}\|\nabla \varphi_{t}(x(t))\|^{2} \nonumber\\
	&\leq& -\dfrac{1}{2}\dot{\varepsilon}(t)\|x^{*}\|^{2}-\dfrac{\delta}{2}\sqrt{\varepsilon(t)}\|\dot{x}(t)\|^{2}- \dfrac{\beta}{2}\|\nabla \varphi_{t}(x(t))\|^{2}
	\end{eqnarray}
Equivalently,
	$$\dfrac{\delta}{2}\sqrt{\varepsilon(t)}\|\dot{x}(t)\|^{2}
+ \dfrac{\beta}{2}\|\nabla \varphi_{t}(x(t))\|^{2}\leq-\dot{\mathcal{E}}_1(t)-\dfrac{1}{2}\dot{\varepsilon}(t)\|x^{*}\|^{2}. $$
By multiplying this last equality by $t^{\frac{3r}{2}-1}$ and integrating on $ [t_1,T],$ we get 
\begin{eqnarray}\label{estintgral}
\dfrac{\delta}{2}\int_{t_1}^{T}t^{r-1}\|\dot{x}(t)\|^{2}dt
+ \dfrac{\beta}{2}\int_{t_1}^{T}t^{\frac{3r}{2}-1}\|\nabla \varphi_{t}(x(t))\|^{2}dt\leq-\int_{t_1}^{T}t^{\frac{3r}{2}-1}\dot{\mathcal{E}}_1(t)dt+\dfrac{\|x^{*}\|^{2}}{2}\int_{t_1}^{T}t^{\frac{r}{2}-2}dt.
\end{eqnarray}
We have 
\begin{eqnarray}\label{key20}
&&-\int_{t_1}^{T}t^{\frac{3r}{2}-1}\dot{\mathcal{E}}_1(t)dt= \left(t_1^{\frac{3r}{2}-1}\mathcal{E}_1(t_1)-t^{\frac{3r}{2}-1}\mathcal{E}_1(t)\right)+\frac{3r-2}{2}\int_{t_1}^{T}t^{\frac{3r}{2}-2}\mathcal{E}_1(t)dt  \nonumber\\
&\leq & t_1^{\frac{3r}{2}-1}\mathcal{E}_1(t_1)+\frac{3r-2}{2}\int_{t_1}^{T}t^{\frac{3r}{2}-2}\left(\varphi_{t}(x(t))-\varphi_{t}(x_{\varepsilon(t)})\right)dt +\frac{3r-2}{4}\int_{t_1}^{T}t^{\frac{3r}{2}-2}\|\dot{x}(t)+\beta\nabla \varphi_{t}(x(t))\|^{2}dt \nonumber\\
&\leq &t_1^{\frac{3r}{2}-1}\mathcal{E}_1(t_1)+\frac{3r-2}{2}\int_{t_1}^{T}t^{\frac{3r}{2}-2}E_{1}(t)dt +\frac{3r-2}{4}\int_{t_1}^{T}t^{\frac{3r}{2}-2}\|\dot{x}(t)+\beta\nabla \varphi_{t}(x(t))\|^{2}dt
\end{eqnarray}
According to \eqref{key10000} and \eqref{vitess}, we deduce that there exists $C>0$ such that 
\begin{eqnarray}\label{key30}-\int_{t_1}^{T}t^{\frac{3r}{2}-1}\dot{\mathcal{E}}_1(t)dt
&\leq &t_1^{\frac{3r}{2}-1}\mathcal{E}_1(t_1)+C\int_{t_1}^{T}t^{r-3}dt.
\end{eqnarray}
From this, we deduce that $-\int_{t_1}^{+\infty}t^{\frac{3r}{2}-1}\dot{\mathcal{E}}_1(t)dt<+\infty.$ By \eqref{estintgral}, we conclude that 
$$\dfrac{\delta}{2}\int_{t_1}^{T}t^{r-1}\|\dot{x}(t)\|^{2}dt
+ \dfrac{\beta}{2}\int_{t_1}^{T}t^{\frac{3r}{2}-1}\|\nabla \varphi_{t}(x(t))\|^{2}dt<+\infty.$$
Therefore 
$$\int_{t_1}^{T}t^{r-1}\|\dot{x}(t)\|^{2}dt
<+\infty\quad\mbox{and}\quad \int_{t_1}^{T}t^{\frac{3r}{2}-1}\|\nabla \varphi_{t}(x(t))\|^{2}dt<+\infty.$$
We  also have
$$ 
\Vert\nabla f(x(t))\Vert^2  \leq \left(\Vert\nabla\varphi_t (x(t))\Vert+\frac{1}{t^{r}}\|x(t)\|\right)^{2} 
\leq   2\Vert\nabla\varphi_t (x(t))\Vert^2+\frac{2}{t^{2r}}\|x(t)\|^{2} .
$$ 
Since $x$ is bounded,  we obtain
$
\int_{t_1}^{+\infty}t^{\frac{3r}{2}-1}\Vert\nabla f(x(t))\Vert^2 dt <+\infty.$
This completes the proof. \qed
\end{proof}

\subsection{System (TRISH)}	
We now come to the corresponding result for (TRISH), stated as a model result in the introduction. 
	\begin{theorem}\label{thm:model-aa}
	Take $\delta>2,$  $\e(t)=\displaystyle\frac{1}{t^{r} } $,\; with $1\leq r<2$.
		Let $x : [t_0, +\infty[ \to \mathcal{H}$ be a solution trajectory of
		\begin{equation}\label{particu2}
		\ddot{x}(t) + \frac{\d}{ \displaystyle{t^{\frac{r}{2}}}}\dot{x}(t) +\beta\nabla^{2} f\left(x(t) \right)\dot{x}(t)+ \nabla f\left(x(t) \right)+ \frac{1}{t^r} x(t)=0.
		\end{equation}
	Then,  we have the following  estimates
		\begin{eqnarray}\label{resut1}
		&& f(x(t))-\min_{\cH} f= \mathcal O \left( \displaystyle\frac{1}{t^{r} }   \right) \mbox{ as } \; t \to +\infty;\\
		&& \label{resut2} \|x(t) -x_{\varepsilon(t)}\|^2=\mathcal{O}\left(\dfrac{1}{ t^{\frac{2-r}2}}\right) \mbox{ as} \; t \to +\infty.\\
		&& \label{resut3} \|\dot{x}(t)+\beta \nabla f(x(t))\|=\mathcal{O}\left(\dfrac{1}{ t^{\frac{r+2}4}}\right) \mbox{ as } \; t \to +\infty.\\
		&& \int_{t_1}^{+\infty}t^{r-1}\Vert\dot{x}(t)\Vert^2 dt<+\infty, \quad \int_{t_1}^{+\infty}t^{\frac{3r -2}{2}}\Vert\nabla f(x(t))\Vert^2 dt<+\infty. \label{grad-est}
		\end{eqnarray}
		\end{theorem}
\begin{proof}
a) 
Taking $\e(t)=\displaystyle\frac{1}{t^{r} } $ in Corollary \ref{corol2} gives
 \eqref{particu2}. So if the condition $( \mathcal{H}_0)$ is satisfied,
we  get 
	\begin{equation}\label{Lyap-basic3b}
E_0(t)\leq \dfrac{\|x^{*}\|^{2}}{2\gamma(t)} \displaystyle \int_{t_1}^{t}\left[\left((\lambda c+2a)\lambda\dfrac{\dot{\varepsilon}^{2}(s)}{\varepsilon^{\frac{3}{2}}(s)}-\dot{\varepsilon}(s)+\beta \lambda(\delta-\lambda)\varepsilon^{2}(t)\right) \gamma(s) \right]ds + \dfrac{\gamma(t_1)E_0(t_1)}{\gamma(t)}.
\end{equation}
As in the proof of  Theorem \ref{thm:model-a}, since $r<2, $  take the parameters $a>1, c>2$, 
$\lambda>0$ such that
\begin{center}
$\frac12\left(\delta+\frac{1}{c}+\sqrt{(\delta+\frac{1}{c})^{2}-2}\right)
< \lambda < \min\left(\frac{a}{a+1}\delta, \frac{\delta+\sqrt{\delta^{2}-4}}{2}\right)  .$
\end{center}
For  $t\geq t_1$ large enough, one can prove that
$$
\dfrac{d}{dt}\left(\dfrac{1}{\sqrt{\varepsilon(t)}}\right) =\frac{r}2t^{\frac{r-2}{2}}
\leq \min\left(\frac{a}{a+1}\delta, \; \frac{\delta+\sqrt{\delta^{2}-4}}{2}\right),
\quad \mbox{and}\quad \beta \delta \leq t^{\frac{r}{2}}.$$
This means that the condition $(\mathcal{H}_0)$ is satisfied.
By setting $\lambda_0:=(\lambda c+2a)\lambda , \; \delta_0:=\dfrac{2(\delta-\lambda)}{2-r}$ and $\beta_0=\beta\lambda\left(\delta-\lambda\right)$ and  combining the equations  \eqref{mu} and \eqref{gamma-est-1} with \eqref{Lyap-basic3b},
we get 
	\begin{equation}\label{Lyap-basic1-b}
	E_{0}(t)\leq \dfrac{r}{2 t^{\frac r2} \exp\left(\delta_0t^{\frac{2-r}2} \right)}\displaystyle     \int_{t_1}^{t}\left( \dfrac{\lambda_0r}{s^2 } + \frac{1}{s^{\frac{r+2}2}}+\dfrac{\beta_0}{rs^{\frac{3r}{2}}}\right)\exp\left(\delta_0s^{\frac{2-r}2}\right)ds+\dfrac{\gamma(t_1)E_0(t_1)}{\gamma(t)}.
	\end{equation}
Let us estimate the integral $\displaystyle     \int_{t_1}^{t}\left( \dfrac{\lambda_0r}{s^2 } + \frac{1}{s^{\frac{r+2}2}}+\dfrac{\beta_0}{rs^{\frac{3r}{2}}}\right)\exp\left(\delta_0s^{\frac{2-r}2}\right)ds.$ For $\rho>0$
	$$
	\dfrac{d}{ds}\left( \dfrac{1}{\rho s  } \exp\left(\delta_0 s^{\frac{2-r}2}\right)\right)= \left(  -\dfrac{1}{\rho s^2 } + \dfrac{\delta_0(2-r)}{2\rho s^{\frac{r+2}2}}   
	\right)\exp\left(\delta_0s^{\frac{2-r}2}\right).
	$$
So, we need to show that 
	\begin{equation}\label{integral}
	\dfrac{\lambda_0r}{s^2 } + \dfrac{1}{s^{\frac{r+2}2}}+\dfrac{\beta_0}{rs^{\frac{3r}{2}}} \leq -\dfrac{1}{\rho s^2 } + \dfrac{\delta_0(2-r)}{2\rho s^{\frac{r+2}2}}.
	\end{equation}
	Since 
	$r\geq 1$, we have for $s$ large enough, $\dfrac{\lambda_0r}{s^2 } + \dfrac{1}{s^{\frac{r+2}2}}+\dfrac{\beta_0}{rs^{\frac{3r}{2}}} \leq \dfrac{\lambda_0r}{s^2 } + \dfrac{1+\frac{\beta_0}{r}}{s^{\frac{r+2}2}}. $
	
\noindent 	By taking $\rho < \frac{r}{(a+1)(r+\beta_0)}\delta,$ we have
\begin{equation*}
\begin{array}{lll}
	\dfrac{\lambda_0r}{s^2 } + \dfrac{1+\frac{\beta_0}{r}}{s^{\frac{r+2}2}} \leq -\dfrac{1}{\rho s^2 } + \dfrac{\delta_0(2-r)}{2\rho s^{\frac{r+2}2}}   
& \Longleftrightarrow &
\dfrac{\lambda_0r+\frac1{\rho}}{s^2 }  \leq \left(\dfrac{\delta_0(2-r)}{2\rho}-1-\frac{\beta_0}{r}\right) \dfrac1{s^{\frac{r+2}2}}
= \dfrac{\frac{\delta-\lambda}{\rho}-1-\frac{\beta_0}{r}}{s^{\frac{r+2}2}}\\
	& \Longleftrightarrow &
\dfrac{\lambda_0r+\frac1{\rho}}{s^{\frac{2-r}2} }  \leq \dfrac{\delta-\lambda}{\rho}-1-\dfrac{\beta_0}{r}.
\end{array}
\end{equation*}
We have $0<r<2$, which means that    $\lim_{s \to +\infty} \dfrac{1}{s^{\frac{2-r}2} } =0$.  Combining the fact that $ \lambda<\frac{a}{a+1}\delta,$ with the choice of $\rho,$  one can check that 
	$$
	\delta-\lambda-\rho(1+\dfrac{\beta_0}{r}) > \underbrace{\left(\frac{a}{a+1}\delta-\lambda\right)}_{>0}+ \frac{1}{a+1}\delta- \rho(1+\dfrac{\beta_0}{r}) > -(1+\dfrac{\beta_0}{r})\left(\rho - \frac{r}{(a+1)(r+\beta_0)}\delta \right)> 0.
	$$
Therefore, for $s$ large enough, the last above inequalities are  satisfied, which implies that, for $1\leq r<2,$ and  $t_1$ large enough, we have  
	\begin{eqnarray*}
	E_{0}(t) &\leq& \dfrac{r}{2 t^{\frac r2} \exp\left(\delta_0t^{\frac{2-r}2} \right)}\displaystyle \int_{t_1}^{t}\left(  -\dfrac{1}{\rho s^2 } + \dfrac{\delta_0(2-r)}{2\rho s^{\frac{r+2}2}}   
+\dfrac{\beta_0}{rs^{\frac{3r}{2}}}`	\right)\exp\left(\delta_0s^{\frac{2-r}2}\right)ds+\dfrac{\gamma(t_1)E_0(t_1)}{\gamma(t)}\\
	&=& \dfrac{1}{2  t^{\frac r2} \exp\left(\delta_0t^{\frac{2-r}2} \right)}\displaystyle\int_{t_1}^{t}\dfrac{d}{ds}\left( \dfrac{1}{\rho s } \exp\left(\delta_0s^{\frac{2-r}2}\right)\right)ds+\dfrac{\gamma(t_1)E_0(t_1)}{\gamma(t)}\\
	&=& \dfrac{r}{ 2\rho t^{\frac{r+2}2}} - \dfrac{r}{ t^{\frac{r}2}\exp\left(\delta_0t^{\frac{2-r}2} \right)}  \dfrac{1}{2\rho t_1 } \exp\left(\delta_0t_1^{\frac{2-r}2}\right)+\dfrac{\gamma(t_1)E_0(t_1)}{\gamma(t)}
	\leq \dfrac{r}{2\rho t^{\frac{r+2}2}}+\dfrac{\gamma(t_1)E_0(t_1)}{\gamma(t)}.
	\end{eqnarray*}	
We now proceed in the same way as in the proof of Theorem \ref{thm:model-a}. Since $\dfrac{\gamma(t_1)E_0(t_1)}{\gamma(t)}$ has also an exponential decay to zero,  we deduce that for $t$ large enough,  \eqref{resut1}, \eqref{resut2} and \eqref{resut3} are satisfied.

\noindent b) We now come to the precise integral estimates of the velocities and gradient terms. Parallel to the study for (TRISHE) we proceed with the Lyapunov function $\mathcal{E}_p$ defined in  \eqref{def:Ep}.	
The system (TRISH) corresponds to $p=0$, so we consider
\begin{center}
$\mathcal{E}_0(t)\eqdef 
\left(\varphi_{t}(x(t))-\varphi_{t}(x_{\varepsilon(t)})\right) +\dfrac{1}{2}\|\dot{x}(t)+\beta \nabla f(x(t))\|^{2}.$
\end{center}
Using successively the derivation chain rule,  the  equation (TRISH), and $\dot{\varepsilon}(t) \leq 0$, we get
$$\begin{array}{lll}
\dfrac{d}{dt}\mathcal{E}_0(t)&=&\langle\nabla\varphi_t(x(t)),\dot{x}(t)\rangle+\frac{\dot{\varepsilon}(t)}{2}\|x(t)\|^{2}+\langle\dot{x}(t)+\beta\nabla f(x(t)),\ddot{x}(t)+\beta \nabla^{2} f(x(t))\dot{x}(t) \rangle \vspace{1mm}\\
& = & -\beta \langle\nabla \varphi_t(x(t)) , \nabla f(x(t))\rangle - \delta\sqrt{\varepsilon(t)}\|\dot{x}(t)\|^{2}+\frac{\dot{\varepsilon}(t)}{2}\|x(t)\|^{2}-\delta\beta\sqrt{\varepsilon(t)}\langle\nabla f(x(t)), \dot{x}(t) \rangle \vspace{1mm}\\
& \leq  & -\beta\|\nabla f(x(t))\|^{2}-\beta \varepsilon(t) \langle x(t), \nabla f(x(t))\rangle - \delta\sqrt{\varepsilon(t)}\|\dot{x}(t)\|^{2}-\delta\beta\sqrt{\varepsilon(t)}\langle\nabla f(x(t)), \dot{x}(t) \rangle \vspace{1mm}\\
& \leq  & -\frac{\beta}{2}\|\nabla f(x(t))\|^{2}+\dfrac{\beta\varepsilon^{2}(t)}{2}\| x(t)\|^{2} - \delta\sqrt{\varepsilon(t)}\|\dot{x}(t)\|^{2}-\delta\beta\sqrt{\varepsilon(t)}
\langle\nabla f(x(t)),  \dot{x}(t) \rangle\\
\end{array}
$$
Equivalently,
$$\frac{\beta}{2}\|\nabla f(x(t))\|^{2} + \delta\sqrt{\varepsilon(t)}\|\dot{x}(t)\|^{2}\leq -\dfrac{d}{dt}\mathcal{E}_0(t)+\dfrac{\beta\varepsilon^{2}(t)}{2}\| x(t)\|^{2}-\delta\beta\sqrt{\varepsilon(t)}\langle\nabla f(x(t)), \dot{x}(t) \rangle.$$
By multiplying this last equality by $t^{\frac{3r}{2}-1}$ and integrating on $ [t_1,T],$ we get 
\begin{eqnarray}
&&\frac{\beta}{2}\int_{t_1}^{T}t^{\frac{3r}{2}-1}\|\nabla f(x(t))\|^{2}dt + \delta\int_{t_1}^{T}t^{r-1}\|\dot{x}(t)\|^{2}dt   \nonumber \\
&& \leq -\int_{t_1}^{T}t^{\frac{3r}{2}-1}\dfrac{d}{dt}
\mathcal{E}_0(t)dt+\beta\int_{t_1}^{T}t^{-\frac{r}{2}-1}\| x(t)\|^{2}dt-\delta\beta\int_{t_1}^{T}t^{r-1}\langle\nabla f(x(t)), \dot{x}(t) \rangle dt. \label{part}
\end{eqnarray}
Let us show that the second member of this last inequality converges when $T$ goes to infinity.

$\bullet$ \; For the first term of second member of inequality \eqref{part}, we have 
\begin{eqnarray}
&&-\int_{t_1}^{T}t^{\frac{3r}{2}-1}\dfrac{d}{dt}\mathcal{E}_0(t)dt = -\left[t^{\frac{3r}{2}-1}\mathcal{E}_0(t)\right]^{T}_{t_1}+\left(\frac{3r}{2}-1\right)\int_{t_1}^{T}t^{\frac{3r}{2}-2}
\mathcal{E}_0(t)dt  \nonumber \\
& \leq  & \underbrace{t_{1}^{\frac{3r}{2}-1}\mathcal{E}_0(t_{1})}_{C_1}+\left(\frac{3r}{2}-1\right)\int_{t_1}^{T}t^{\frac{3r}{2}-2}\left(\varphi_t(x(t))-f(x^{*})\right)dt 
+\left(\frac{3r-2}{4}\right)\int_{t_1}^{T}t^{\frac{3r}{2}-2}\|\dot{x}(t)+\beta \nabla f(x(t))\|^{2}dt  \nonumber \\
&\leq&   C_1+\left(\frac{3r-2}{2}\right)\int_{t_1}^{T}t^{\frac{3r}{2}-2}\left(f(x(t))-f(x^{*})\right)dt+\left(\frac{3r-2}{4}\right)\int_{t_1}^{T}t^{\frac{r}{2}-2}\|x(t)\|^{2}dt \nonumber \\
&&+\left(\frac{3r-2}{4}\right)\int_{t_1}^{T}t^{\frac{3r}{2}-2}\|\dot{x}(t)+\beta \nabla f(x(t))\|^{2}dt .\label{new-Lyap-grad}
\end{eqnarray}
Using \eqref{resut3} and \eqref{resut1}, and the fact that $x(\cdot)$ is bounded, we obtain
\begin{equation}\label{part1}
\begin{array}{lll}
-\int_{t_1}^{+\infty}t^{\frac{3r}{2}-1}\dfrac{d}{dt}\mathcal{E}_0(t)dt& \leq  & C_1+C_2\int_{t_1}^{+\infty}t^{\frac{r}{2}-2}dt+C_3\int_{t_1}^{+\infty}t^{r-3}dt<+\infty,\quad \mbox{because}\; r<2 .
\end{array}
\end{equation}

$\bullet$ \; Consider now the third term of the second member of the inequality \eqref{part}. According to the equality  $\dfrac{d}{dt}\left(f(x(t))-f(x^{*})\right)=\langle \nabla f(x(t)), \dot{x}(t)\rangle,$ we have
\begin{eqnarray*}
-\delta\beta\int_{t_1}^{T}t^{r-1}\langle\nabla f(x(t)),   \dot{x}(t) \rangle dt& = &-\delta\beta\left[t^{r-1}\left(f(x(t))-f(x^{*})\right)\right]^{T}_{t_1} \nonumber\\
&+& \delta\beta(r-1)\beta\int_{t_1}^{T}t^{r-2}\left(f(x(t))-f(x^{*})\right)dt  \\
&  \leq & \underbrace{\delta\beta t_{1}^{r-1}\left(f(x(_1))-f(x^{*})\right)}_{C_4} + \delta\beta(r-1)\beta\int_{t_1}^{T}t^{r-2}\left(f(x(t))-f(x^{*})\right)dt. \nonumber
\end{eqnarray*}
 Using \eqref{resut1}, we get 
\begin{equation}\label{part2}
\begin{array}{lll}
-\delta\beta\int_{t_1}^{+\infty}t^{r-1}\langle\nabla f(x(t)),  \dot{x}(t) \rangle dt
&  \leq & C_4 + C_5\int_{t_1}^{+\infty}t^{-2}dt<+\infty.
\end{array}
\end{equation}
Collecting \eqref{part}, \eqref{part1} and \eqref{part2}, we conclude that
\begin{equation*}
\frac{\beta}{2}\int_{t_1}^{+\infty}t^{\frac{3r}{2}-1}\|\nabla f(x(t))\|^{2}dt + \delta\int_{t_1}^{+\infty}t^{r-1}\|\dot{x}(t)\|^{2}dt\leq C+\beta\int_{t_1}^{+\infty}t^{-\frac{r}{2}-1}\| x(t)\|^{2}dt .
\end{equation*}
Using again that $x$ is bounded, we deduce that
$$\frac{\beta}{2}\int_{t_1}^{+\infty}t^{\frac{3r}{2}-1}\|\nabla f(x(t))\|^{2}dt + \delta\int_{t_1}^{+\infty}t^{r-1}\|\dot{x}(t)\|^{2}dt<+\infty.$$	
This completes the proof.	\qed
\end{proof}

\section{Numerical illustrations}\label{num}

Let us illustrate our results with the following examples where the function $f$ is taken successively strictly convex, then convex with a continuum of solutions. In a third example, we compare the two systems (TRISH) and (TRISHE). 
The following numerical experiences 
describe in these three situations the behavior of the trajectories generated by the system  (TRIGS) (without the Hessian driven damping) and by the systems (TRISH) and (TRISHE) (with the Hessian driven damping).   All these systems take into account the effect of  Tikhonov regularization. They are differentiated by the presence, or not, of the Hessian driven damping. According to the model situation described in Theorem \ref{thm:model-a} and Theorem  \ref{thm:model-aa}, the Tikhonov regularization parameter is taken equal to $\e(t)=\displaystyle t^{-r}  $, with $ 0<r\leq 2$. We consider different values of the parameter $r$ which plays a key role in  tuning  the viscosity and Tikhonov parameters. We pay particular attention to the case $r$ close to the value 2, which provides fast convergence results. The corresponding dynamical systems are given by:
\begin{eqnarray*}
&&{\rm (TRIGS)} \quad \quad \ddot{x}(t) + \delta  \displaystyle{t^{-\frac{r}{2}}}\dot{x}(t) + \nabla f\left(x(t) \right)+ \displaystyle{t^{-r}}x(t)=0\\
&&{\rm (TRISH)} \quad\quad  \ddot{x}(t) +  \delta\displaystyle{t^{-\frac{r}{2}}}\dot{x}(t)  + \beta \nabla^2 f\left(x(t) \right)\dot x(t) + \nabla f\left(x(t) \right)+ \displaystyle{t^{-r}} x(t)=0 \vspace{2mm}\\
&&{\rm (TRISHE)} \quad\; \ddot{x}(t)   +  \delta \displaystyle{t^{-\frac{r}{2}}}\dot{x}(t) + \beta \nabla^2 f\left(x(t) \right)\dot x(t) - r\displaystyle{t^{-r-1}}x(t)+ \displaystyle{t^{-r}}\dot x(t)  + \nabla f\left(x(t) \right)+ \frac{1}{t^{r}} x(t)=0 .
\end{eqnarray*}
 We choose $\delta=3$, $\beta=1$.\\
To facilitate the comparison of the trajectories corresponding to different dynamics, for example (TRIGS) and (TRISH), they are represented respectively by continuous lines and dotted lines.
All our numerical tests were implemented in Scilab version 6.1 as an open source software.

\begin{example}\label{exemple1} 
Take $f_1  :  ]-1,+\infty[^{2} \to \R$ which is defined by $$f_1(x)=(x_1+x_2^2)-2\ln(x_1+1)(x_2+1).$$
 The function $f_1$ is strictly convex with 
$$
\nabla f_1(x)=\begin{bmatrix}
1-\frac2{x_1+1}\\
2x_2-\frac2{x_2+1}
\end{bmatrix} \text{ and } 
\nabla^2 f_1(x)=\begin{bmatrix}
\frac2{(x_1+1)^2} & 0\\
0 & 2+\frac2{(x_2+1)^2}
\end{bmatrix} .
$$
The unique minimum of $f_1$ is $x^*=(1,(\sqrt 5-1)/2)$. The corresponding trajectories to the systems ${\rm (TRIGS)}, {\rm (TRISHE)}$ are depicted in Figure \ref{fig:trigs-c}.
\begin{figure} 
 \includegraphics[scale=0.35]{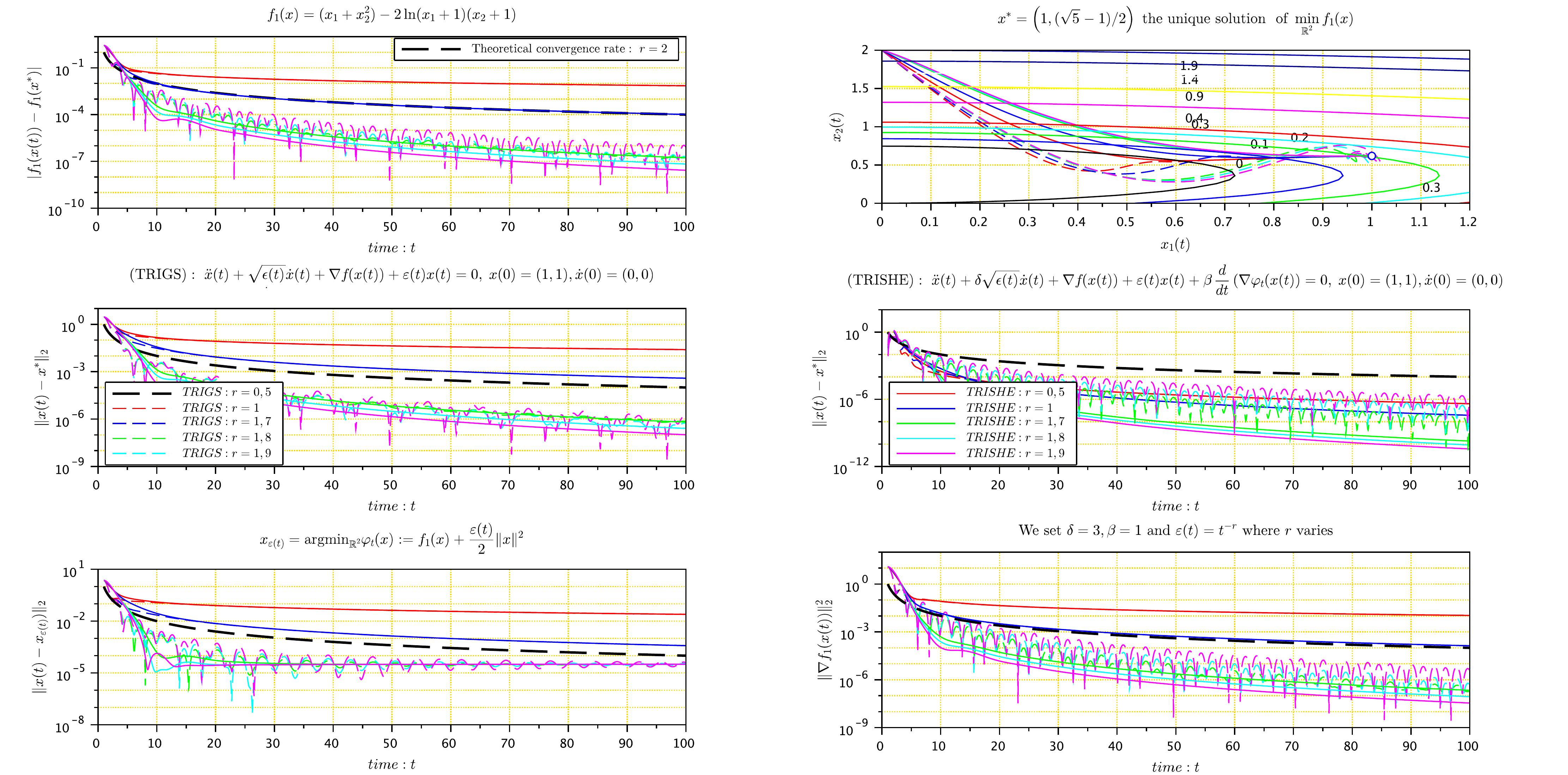}
  \caption{Convergence rates of
values $f_1(x(t))-f(x^*)$,  of trajectories $\|x(t)-x^*\|_2 $, and gradients  \; $\|\nabla f_1(x(t)\|_2$.}
 \label{fig:trigs-c} 
\end{figure}
\end{example}

\begin{example}\label{exemple2} Consider the convex  function $f_2  : \R^{2} \to \R$ defined by 
$$f_2(x)=\frac12(x_1+x_2-1)^2 .$$
 We have
$$
\nabla f_2(x)=\begin{bmatrix}
x_1+x_2-1\\
x_1+x_2-1
\end{bmatrix} \text{ and } 
\nabla^2 f_1(x)=\begin{bmatrix}
1 & 1\\
1 & 1
\end{bmatrix} .
$$
We have  $S_2= \argmin f_2= \{(x_1,1-x_1): x_1\in\R\}$ and $x^*=(\frac12,\frac12)$ is the minimum norm solution. The corresponding trajectories to the systems ${\rm (TRIGS)}, {\rm (TRISHE)}$ are depicted in Figure \ref{fig:trigs-convex}.
\begin{figure} 
  \includegraphics[scale=0.35]{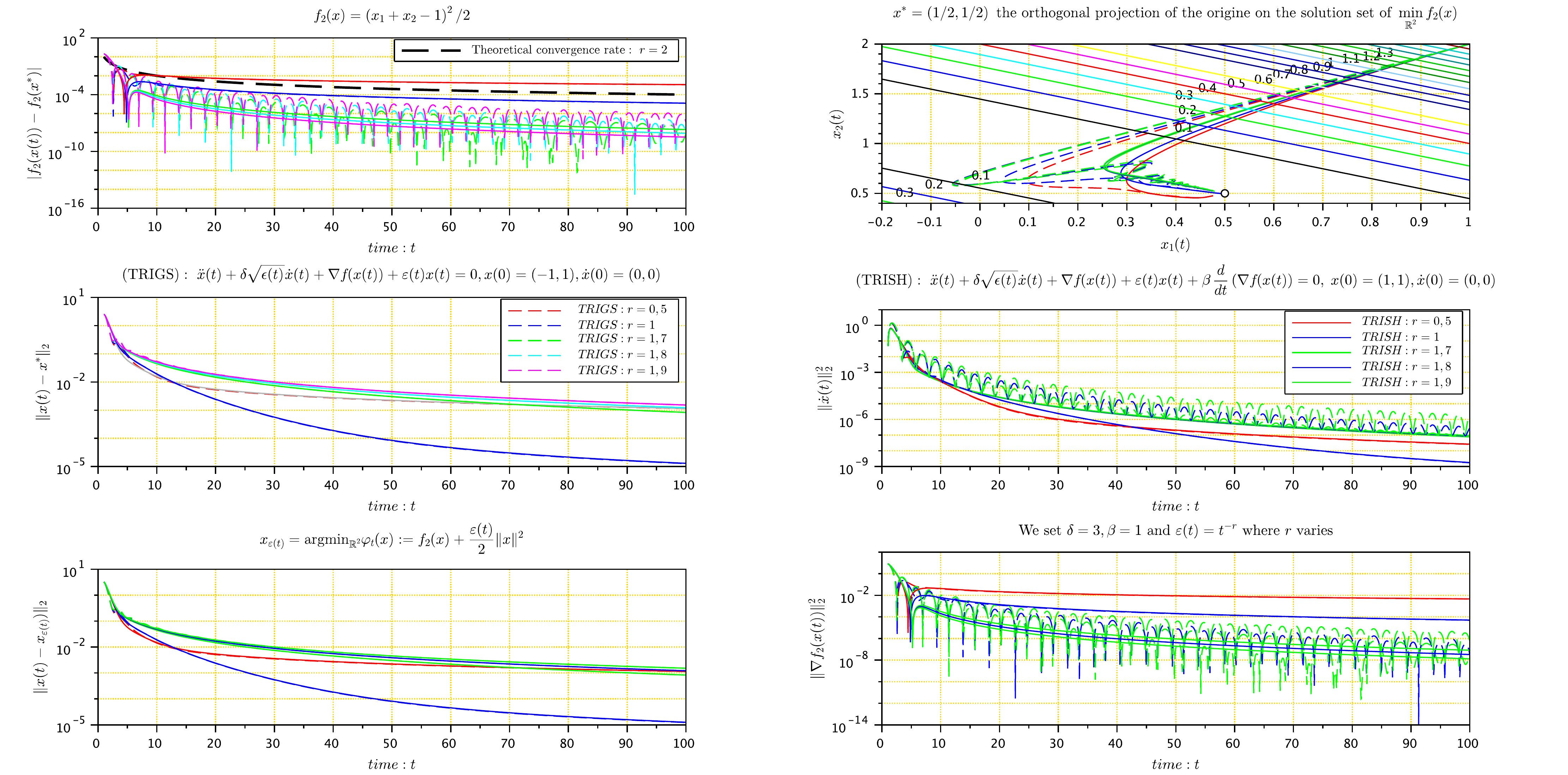}
  \caption{ Convergence rates of
values $f_2(x(t))-f(x^*)$, trajectories $\|x(t)-x^*\|_2$, and  gradients $\|\nabla f_2(x(t)\|_2$.}
 \label{fig:trigs-convex} 
\end{figure}
\end{example}

\begin{example}\label{exemple3} 
To compare the systems (TRISH) and  (TRISHE), we take the same function $f_2$ as in the previous example.
The corresponding trajectories  are depicted in Figure \ref{fig:trish-trishe-convex}.
\begin{figure} 
  \includegraphics[scale=0.35]{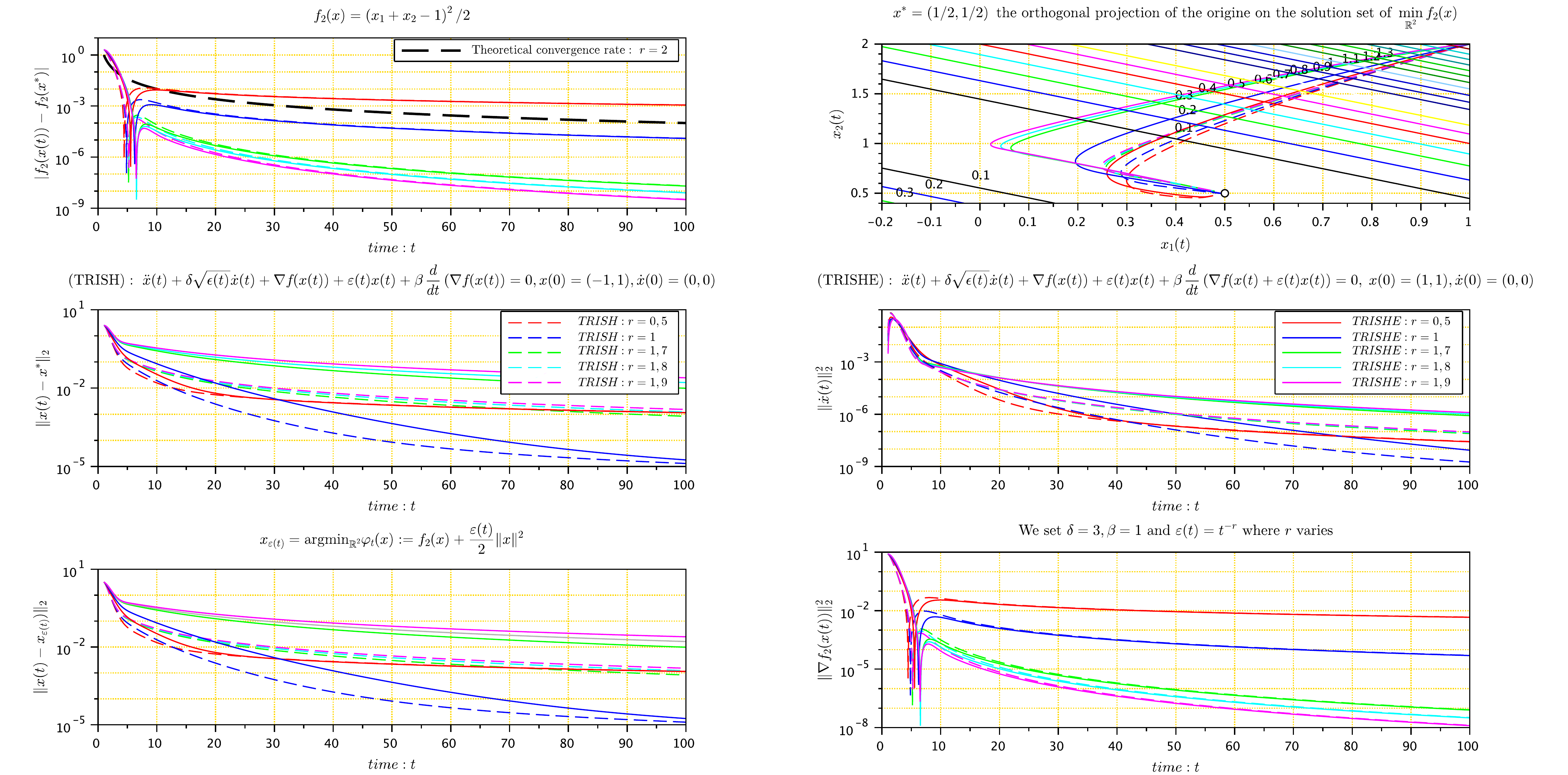}
  \caption{ Comparison of (TRISH) and (TRISHE)}
 \label{fig:trish-trishe-convex} 
\end{figure}
\end{example} 

As predicted by the theory, it is observed that the trajectories generated by  the systems (TRISH) and (TRISHE) have at the same time several remarkable properties: they ensure  fast convergence of the values, fast convergence of the gradients towards zero, and convergence to the  minimum norm solution. The presence of the Hessian driven damping in these dynamics induces a significant attenuation of oscillations (by comparison with (TRIGS)).
The third example shows that the trajectories generated by (TRISH) and (TRISHE) share a very similar behaviour.
We see the advantage of taking $r$ close to $2$ in the presence of the Hessian driven damping. Indeed, $r=2$ gives viscous damping similar  to that of the accelerated gradient method of Nesterov, in which case we  know that the adjustment of the  coefficient $\delta$ plays a crucial role. Note the criticality of the case $r=2$, since for $r<2$ the condition for $\delta$ is $\delta >2$, whereas for $r=2$ we know that we must take $\delta \geq 3$ to get fast convergence.
This is an interesting subject for further research.

\section{Existence of solution trajectories for  (TRISHE)}\label{sec:nonsmooth}

Let us start by establishing the equivalence between the inertial dynamic with Hessian driven damping (TRISHE)
 and a first order system in time and space in the case of a smooth function $f$.
Similar result was first obtained 
in      \cite{AABR},   with applications to mechanics \cite{AMR} and deep learning  \cite{BCPF}. 
\begin{theorem}\label{Thm-first-order-system-existence}
Let $f:\cH\to\R$ be a convex $\cC^2$ function. Suppose that $\alpha\geq0$, $\beta>0$. Let $(x_0,\dot x_0)\in \cH \times \cH$. The 
following statements are equivalent:
\begin{enumerate}[label=\arabic*.]

\item \label{item:Thm-first-order-system_1}
$x:[t_0,+\infty [ \to \cH$ is a solution trajectory of 
\begin{equation}\label{basic-equ-b}
{\rm(TRISHE)} \qquad  \ddot{x}(t)+\delta\sqrt{\varepsilon(t)}\dot{x}(t)+\beta\dfrac{d}{dt}\left(\nabla \varphi_{t}(x(t))\right)+\nabla\varphi_{t}(x(t))=0.
\end{equation}
with the initial conditions $x(t_0)=x_0$, $\dot x(t_0)=\dot x_0$. 

\smallskip

\item \label{item:Thm-first-order-system_2}
$(x,y):[t_0,+\infty [ \to \cH \times \cH$ is a solution trajectory of the first-order system 
\begin{equation}\label{eq:ISEHDequiv1stode}
\begin{cases}
\dot x(t) +  \beta  \nabla \varphi_{t}(x(t)) - \pa{\frac{1}{\beta} - \delta\sqrt{\varepsilon(t)}} x(t) + \frac{1}{\beta}y(t) &= 0 \vspace{1mm}\\
\dot{y}(t)-\pa{\frac{1}{\beta} - \delta\sqrt{\varepsilon(t)} - \frac{\beta \delta}{2} \frac{\dot{\epsilon}(t)}{\sqrt{\varepsilon(t)}      } }x(t) + \frac{1}{\beta} y(t) &= 0,
\end{cases}
\end{equation}
\end{enumerate}
with initial conditions $x(t_0)=x_0$, 
$y(t_0)=-\beta(\dot x_0+\beta \nabla \varphi_{t_0}(x_0))+(1-\delta\sqrt{\varepsilon(t_0)})x_0 $.
\end{theorem}
\begin{proof}
\textit{{\ref{item:Thm-first-order-system_2} }}$\Rightarrow$ \textit{{\ref{item:Thm-first-order-system_1} }} 
Differentiating the first equation of \eqref{eq:ISEHDequiv1stode} gives
\begin{equation}\label{eq:fos01}
\ddot x(t)  +\beta\dfrac{d}{dt}\left(\nabla \varphi_{t}(x(t))\right)+ \delta \frac{\dot{\epsilon}(t)}{2\sqrt{\varepsilon(t)}} x(t)      - \pa{\frac{1}{\beta} - \delta\sqrt{\varepsilon(t)}} \dot x(t)+\frac{1}{\beta}\dot y(t)  =  0 .
\end{equation}
Replacing $\dot y(t)$ by its expression as given  by the second equation of \eqref{eq:ISEHDequiv1stode} gives
\begin{eqnarray}
&&\ddot x(t) + \beta\dfrac{d}{dt}\left(\nabla \varphi_{t}(x(t))\right)  +\delta \frac{\dot{\epsilon}(t)}{2\sqrt{\varepsilon(t)}} x(t)      - \pa{\frac{1}{\beta} - \delta\sqrt{\varepsilon(t)}} \dot x(t)   \nonumber \\
&&+\frac{1}{\beta}\pa{\pa{ \frac{1}{\beta} - \delta\sqrt{\varepsilon(t)} - \frac{\beta \delta}{2} \frac{\dot{\epsilon}(t)}{\sqrt{\varepsilon(t)}      } } x(t)
 -\frac{1}{\beta} y(t)}  =  0 .\label{eq:fos02}
\end{eqnarray}
Then replace $y(t)$ by its expression as given  by the first equation of \eqref{eq:ISEHDequiv1stode}
\begin{eqnarray*}
&&\ddot x(t) + \beta\dfrac{d}{dt}\left(\nabla \varphi_{t}(x(t))\right)  +\delta \frac{\dot{\epsilon}(t)}{2\sqrt{\varepsilon(t)}} x(t)      - \pa{\frac{1}{\beta} - \delta\sqrt{\varepsilon(t)}} \dot x(t)   \nonumber \\
&&+\frac{1}{\beta} \pa{ \frac{1}{\beta} - \delta\sqrt{\varepsilon(t)} - \frac{\beta \delta}{2} \frac{\dot{\epsilon}(t)}{\sqrt{\varepsilon(t)}      } } x(t)
 +\frac{1}{\beta} \left(  \dot x(t) +  \beta  \nabla \varphi_{t}(x(t)) - \pa{\frac{1}{\beta} - \delta\sqrt{\varepsilon(t)}} x(t)   \right)    =  0 .
\end{eqnarray*}
After simplification of the above expression, we obtain \eqref{basic-equ-b}.

\medskip

\textit{{\ref{item:Thm-first-order-system_1}}} $\Rightarrow$ \textit{{\ref{item:Thm-first-order-system_2}}} Define $y(t)$ by  the first equation of \eqref{eq:ISEHDequiv1stode}. Differentiating $y(t)$ and using equation \eqref{basic-equ-b} allows one to eliminate $\ddot{x}(t)$, which finally gives the second equation of \eqref{eq:ISEHDequiv1stode}. \qed
\end{proof}
Based on Theorem~\ref{Thm-first-order-system-existence}, the following first order formulation helps give meaning to the (TRISHE) system  when $f \in \Gamma_0(\cH)$.
It is obtained by substituting the subdifferential $\partial \varphi_{t}$ for the gradient $\nabla \varphi_{t}$ in the first-order formulation \eqref{eq:ISEHDequiv1stode}.

\begin{definition}
Let $\delta >0$, $\beta>0$ and $f \in \Gamma_0(\cH)$. {Given
$(x_0, y_0) \in  \dom(f) \times \cH $, the Cauchy problem for the  inertial system (TRISHE)  with  generalized Hessian driven damping is defined by}
\begin{equation}\label{eq:fos2}
\begin{cases}
\dot x(t) +  \beta  \partial \varphi_{t}(x(t)) - \pa{\frac{1}{\beta} - \delta\sqrt{\varepsilon(t)}} x(t) + \frac{1}{\beta}y(t) & \ni 0 \\
\dot{y}(t)-\pa{\frac{1}{\beta} - \delta\sqrt{\varepsilon(t)} - \frac{\beta \delta}{2} \frac{\dot{\epsilon}(t)}{\sqrt{\varepsilon(t)}      } }x(t) + \frac{1}{\beta} y(t) & = 0 \\
x(t_0)=x_0, y(t_0)=y_0 .
\end{cases}
\end{equation}
\end{definition}
Let us formulate \eqref{eq:fos2} in a condensed form as an evolution equation in the product space $\cH \times \cH$.
Setting $Z(t) = (x(t), y(t)) \in  \cH \times \cH$, \eqref{eq:fos2}
can be equivalently written 

\begin{equation}\label{eq:fos3}
\dot{Z}(t) + \partial \cG(t,Z(t)) + \cD(t, Z(t)) \ni 0, \quad {Z(t_0)=(x_0,y_0)},
\end{equation}

\noindent where $\cG (,\cdot) \in \Gamma_0(\cH \times \cH)$ is the function defined by $\cG (t,Z) = \beta \varphi_{t}(x)$, and the time-dependent operator $\cD:\ [t_0,+\infty[ \times \cH \times \cH \to \cH \times \cH$ is given by
\begin{equation}\label{eq:fos5}
\cD(t,Z)=
  \pa{- \pa{\frac{1}{\beta} - \delta\sqrt{\varepsilon(t)}} x + \frac{1}{\beta}y,
    -\pa{\frac{1}{\beta} - \delta\sqrt{\varepsilon(t)} - \frac{\beta \delta}{2} \frac{\dot{\epsilon}(t)}{\sqrt{\varepsilon(t)}      } }x + \frac{1}{\beta} y }.
\end{equation}

The differential inclusion \eqref{eq:fos3} is governed by the sum of the time dependent maximally monotone operator $\partial \cG (t,.\cdot)$ (a convex subdifferential) and the time-dependent linear continuous operator $\cD (t,\cdot)$. The existence and uniqueness of a global solution for the corresponding Cauchy problem is a consequence of the general theory of evolution equations governed by maximally monotone operators \cite[Proposition 3.12]{Bre1}, and of the fact that $t \mapsto \varphi_t (x)$ is a nonincreasing function, see \cite{AD}. In this setting, the notion of classical solution is replaced by the  notion of strong solution, see \cite[Definition 3.1]{Bre1}, \cite[Theorem 4.4]{APR}, \cite[Theorem 2.4]{AFK}.

\section{Conclusion, perspective}\label{sec:Conclusion}

For convex optimization in Hilbert spaces, we have introduced a damped inertial dynamics which combines Hessian driven damping with Tikhonov regularization. 
The Hessian driven damping and the Tikhonov regularization term  induce specific favorable geometric properties, related to curvature aspects.
The Tikhonov term regulates the objective function. It makes the  dynamic relevant of the heavy ball with friction method for a strongly convex function.
The Hessian driven damping acts on the velocity vector in a similar way as continuous Newton's method. It has a corrective effect by damping the oscillations that arise with ill-conditioned optimization problems.
It turns out that the two techniques combine well and provide a substantial improvement to Nesterov's accelerated gradient method.
While preserving fast convergence of values, they ensure fast convergence of gradients to zero, they significantly reduce oscillations, and  provide convergence to the minimum norm solution.

Our study provides a solid basis for the convergence analysis of algorithms obtained by temporal discretization, which is a subject of further work.
Our approach  calls for many developments. We showed that our approach can be naturally extended to the case of nonsmooth convex optimization, and   the study of additively structured "smooth + nonsmooth" convex optimization problems.
Our study naturally leads to applications in various fields such as inverse problems for which  strong convergence of trajectories, and obtaining a solution close to a desired state are key properties. 

It is likely that a parallel approach can be developed for   multiobjective optimization for the dynamical approach to Pareto optima, and within the framework of potential games. The Lyapunov analysis developed in this paper could  also be very  useful to study the asymptotic stabilization  of several classes of PDE's, for example nonlinear damped wave equations.
One of the main challenges related to our study is whether similar convergence results can be obtained using autonomous systems, see \cite{ABotCest}  for a first systematic study of this question. 
Indeed the study of autonomous versions of the Tikhonov method, such as the Haugazeau method, in the context of dynamic systems, and rapid optimization, is a field largely to be explored.

\paragraph{\textbf{Acknowledgments}:} The research of A\"icha BALHAG was supported by  the EIPHI Graduate School (contract ANR-17-EURE-0002).


\begin{thebibliography}{10}

\bibitem{AAS} {\sc B. Abbas, H. Attouch, B. F. Svaiter},    Newton-like dynamics and forward-backward methods for structured monotone inclusions
in Hilbert spaces, J. Optim. Theory Appl., 161 (2) (2014),  331-360.

\bibitem{AAV-algo}  {\sc S. Adly, H. Attouch, Van Nam Vo}, Newton-type inertial algorithms for solving monotone equations
governed by sums of potential and nonpotential operators,
Applied Mathematics and Optimization (AMOP), 2021. hal-03260201.


\bibitem{AABR} {\sc F. Alvarez, H. Attouch, J. Bolte, P. Redont}, A second-order gradient-like dissipative dynamical system with Hessian-driven damping. Application to optimization and mechanics,
    J. Math. Pures Appl.,  81 (8) (2002),  747--779.

\bibitem{AlvCab} {\sc F. Alvarez, A. Cabot},  Asymptotic selection of viscosity equilibria of semilinear evolution equations by the introduction of a slowly vanishing term,  Discrete Contin. Dyn. Syst. 15 (2006),  921--938.


\bibitem{AAD1}{\sc V. Apidopoulos, J.-F. Aujol,  Ch. Dossal},
 The differential inclusion modeling the FISTA algorithm and optimality of convergence rate in the case $b \leq 3$, SIAM J. Optim.,  28 (1)  (2018),  551--574.


\bibitem{Att2} {\sc H. Attouch},  Viscosity solutions of minimization problems, SIAM J. Optim. 6 (3) (1996), 769--806.

\bibitem{ABCR} {\sc H. Attouch, A. Balhag, Z. Chbani, H. Riahi},
Damped inertial dynamics with vanishing Tikhonov regularization: Strong asymptotic convergence towards the minimum norm solution,
J.  Differential Equations, 311 (2022),  29--58.

\bibitem{ABotCest}  {\sc H. Attouch, R.I. Bo\c t,  E.R. Csetnek},
Fast optimization via  inertial dynamics  with closed-loop damping,
Journal of the European Mathematical Society (JEMS), 2021, hal-02910307.






\bibitem{AC10} {\sc H. Attouch, A.  Cabot},  Asymptotic stabilization of inertial gradient dynamics with time-dependent viscosity,   J. Differential Equations, 263 (9), (2017), 5412--5458.



\bibitem{ACFR} {\sc  H. Attouch, Z. Chbani, J. Fadili, H. Riahi}, First order optimization algorithms via inertial  systems with Hessian driven damping,  Math. Program. (2020),  https://doi.org/10.1007/s10107-020-01591-1.

\bibitem{ACFR-Optimisation}{\sc H. Attouch,  Z. Chbani, J. Fadili, H. Riahi}, Convergence of iterates for first-order optimization algorithms with inertia and Hessian driven damping, Optimization, 
 https://doi.org/10.1080/02331934.2021.2009828.  (2021).


\bibitem{ACPR}
{\sc H. Attouch, Z. Chbani, J. Peypouquet, P. Redont}, Fast convergence of inertial dynamics and algorithms with asymptotic vanishing viscosity, Math. Program., 168 (1-2) (2018),  123--175.


\bibitem{ACR} {\sc H. Attouch,  Z. Chbani, H. Riahi},
 Combining fast inertial dynamics for convex optimization with Tikhonov regularization,
J. Math. Anal. Appl, 457 (2018),  1065--1094.


\bibitem{AttCom} {\sc H. Attouch,  R. Cominetti},  A dynamical approach to convex
minimization coupling approximation with the steepest descent method,   J.
Differential Equations, 128 (2) (1996), 519--540.



\bibitem{AttCza1}{\sc H. Attouch, M.-O. Czarnecki},  Asymptotic control and stabilization
of nonlinear oscillators with non-isolated equilibria, J. Differential Equations 179 (2002), 278--310.


\bibitem{AttCza2} {\sc H. Attouch, M.-O. Czarnecki},  Asymptotic behavior of coupled dynamical systems with multiscale aspects, J. Differential Equations 248 (2010), 1315--1344.






\bibitem{Att-Czar-last} {\sc H. Attouch, M.-O. Czarnecki},  Asymptotic behavior of gradient-like dynamical systems involving inertia and multiscale aspects, J. Differential Equations,  262 (3)  (2017),  2745--2770.

\bibitem{AD} {\sc H. Attouch, A. Damlamian},  Strong solutions for parabolic variational inequalities,
  Nonlinear Analysis, TMA, 2 (3) (1978),    329-353.

 \bibitem{AF} {\sc H. Attouch,  J. Fadili},
 From the Ravine method to the Nesterov method and vice
versa: a dynamic system perspective, arXiv:2201.11643v1 [math.OC] 27 Jan 2022.


\bibitem{AFK} {\sc H. Attouch,  J. Fadili, V. Kungurtsev},
 On the effect of perturbations, errors in first-order optimization
methods with inertia and Hessian driven damping,
arXiv:2106.16159v1 [math.OC] 30 Jun 2021.

\bibitem{AL} {\sc H. Attouch, S. L\'aszl\'o},  Convex optimization via inertial algorithms with vanishing Tikhonov  regularization: fast convergence to the minimum norm solution,
arXiv:2104.11987v1 [math.OC] 24 Apr 2021.




\bibitem{AMR} {\sc H. Attouch, P.E. Maing\'e, P. Redont},  A second-order differential system with Hessian-driven damping; Application to non-elastic shock laws,
  Differential Equations and Applications,  4 (1) (2012), 27--65.
  
  \bibitem{AMAS}  {\sc H. Attouch, M. Marques Alves, B.F. Svaiter},
A dynamic approach to a proximal-Newton method for monotone inclusions in Hilbert Spaces, with complexity $\mathcal O(1/n^2)$,  J. of Convex Analysis, 23 (1) (2016),  139--180.


\bibitem{AP} {\sc H. Attouch,  J. Peypouquet},  The rate of convergence of Nesterov's accelerated  forward-backward method is actually faster than $1/k^2$, SIAM J. Optim., 26 (3) (2016),  1824--1834.


\bibitem{APR} {\sc H. Attouch, J. Peypouquet, P. Redont},   Fast convex minimization via inertial dynamics with Hessian driven damping,   J. Differential Equations, 261 (10),  (2016),  5734--5783.

\bibitem{AS} {\sc H. Attouch, B. F. Svaiter},
  A continuous dynamical Newton-Like approach to solving monotone inclusions,  SIAM J. Control Optim., 49 (2)  (2011),  574--598.


\bibitem{BaiCom} {\sc J.-B. Baillon, R. Cominetti},  A convergence result for non-autonomous subgradient evolution equations and its application to the steepest descent exponential penalty trajectory in linear programming,  J. Funct. Anal. 187 (2001), 263--273.



\bibitem{BC}{\sc H. Bauschke, P. L. Combettes},   Convex Analysis and Monotone Operator Theory in Hilbert spaces, CMS Books in Mathematics, Springer,   (2011).


\bibitem{BCPF} {\sc J. Bolte, C. Castera, E. Pauwels,  
C. F\'evotte}, An Inertial Newton Algorithm for Deep Learning, Journal of Machine Learning Research, 22 (2021), 1--31.




\bibitem{BCL}
{\sc R. I. Bo\c t, E. R. Csetnek, S.C. L\'aszl\'o},     Tikhonov regularization of a second order dynamical system with Hessian damping, Math. Program., 189 (2021), 151--186. 

\bibitem{Bre1}{\sc H. Br\'ezis},   Op\'erateurs maximaux monotones dans les espaces de Hilbert et \'equations d'\'evolution, Lecture Notes 5, North Holland, (1972).

\bibitem{Bre2}{\sc H. Br\'ezis},   Analyse fonctionnelle, Masson, 1983.

 \bibitem{Cabot-inertiel}{\sc A. Cabot},  Inertial gradient-like dynamical system controlled by a stabilizing term, J. Optim. Theory Appl., 120 (2004),  275--303.

\bibitem{Cab} {\sc A. Cabot},  Proximal point algorithm controlled by a slowly vanishing term: Applications to hierarchical minimization, SIAM J. Optim.,  15 (2) (2005), 555--572.


 \bibitem{CEG}{\sc  A. Cabot, H. Engler, S. Gadat},  On the long time behavior of second order differential equations
with asymptotically small dissipation,
Trans. Amer. Math. Soc.,  361 (2009), 5983--6017.

\bibitem{CD}{\sc  A. Chambolle, Ch. Dossal},  On the convergence of the iterates of Fista,
J. Opt. Theory Appl., 166 (2015), 968--982.


\bibitem{CPS} {\sc R. Cominetti, J. Peypouquet, S. Sorin},  Strong asymptotic convergence of evolution equations governed by maximal monotone operators with Tikhonov regularization,  J. Differential Equations, 245 (2008),  3753--3763.






\bibitem{Hirstoaga}{\sc S.A. Hirstoaga},
 Approximation et r\'esolution de probl\`emes d'\'equilibre, de point fixe et
d'inclusion monotone.  PhD thesis, Universit\'e Pierre et Marie Curie - Paris VI, 2006, HAL Id: tel-00137228.


\bibitem{JM-Tikh}{\sc M.A. Jendoubi, R. May},  On an asymptotically autonomous system with Tikhonov type
regularizing term, Archiv der Mathematik, 95 (4) (2010),  389--399.


\bibitem{MME} {\sc  R. May, C. Mnasri, M. Elloumi}, Asymptotic for a second order evolution equation with damping and regularizing terms, AIMS Mathematics, 6 (5) (2021), 4901--4914.

\bibitem{Nest1}{\sc  Y. Nesterov},   A method of solving a convex programming problem with convergence rate
$O(1/k^2)$, Soviet Math. Dokl.,  27 (1983), 372--376.



\bibitem{Nest2}{\sc  Y. Nesterov},  Introductory lectures on convex optimization: A basic course, volume 87 of
Applied Optimization. Kluwer Academic Publishers, Boston, MA, 2004.


\bibitem{Nest3}{\sc  Y. Nesterov},  How to Make the Gradients Small, Optima 88, Mathematical Optimization Society Newsletter, May 2012.


\bibitem{Polyak-64}{\sc  B. Polyak},
Some methods of speeding up the convergence of iteration methods, USSR Computational Mathematics and Mathematical Physics, 4 (1964), 1--17.

\bibitem{Polyak}{\sc  B. Polyak},     Introduction to Optimization, New York, NY: Optimization Software-Inc,  1987.


\bibitem{Siegel} {\sc W. Siegel},  Accelerated first-order methods: Differential equations and Lyapunov functions,\\ arXiv:1903.05671v1 [math.OC], 2019.


\bibitem{SDJS}{\sc B. Shi, S.S. Du,  M. I. Jordan,  W. J. Su},
 Understanding the acceleration phenomenon via high-resolution differential equations,   Math. Program. (2021). https://doi.org/10.1007/s10107-021-01681-8.


\bibitem{SBC}{\sc W.  Su,  S. Boyd,  E. J. Cand\`es}, 
A Differential Equation for Modeling Nesterov's
Accelerated Gradient Method: Theory and Insights.
J. Mach. Learn. Res., 17(153), (2016), 1--43.



\bibitem{Tikh}{\sc A. N. Tikhonov}, Doklady Akademii Nauk SSSR 151 (1963) 501--504, (Translated in "Solution of incorrectly formulated problems and
the regularization method". Soviet Mathematics 4 (1963) 1035--1038.


\bibitem{TA}{\sc A. N. Tikhonov,  V. Y. Arsenin},  Solutions of Ill-Posed Problems,
Winston, New York, 1977.



\bibitem{Torralba}{\sc D. Torralba},  Convergence epigraphique et changements d'\'echelle en analyse variationnelle et optimisation,
PhD thesis, Universit\'e Montpellier, 1996.


\bibitem{WRJ} {\sc A. C. Wilson, B. Recht, M. I. Jordan}, 
A Lyapunov analysis of momentum methods in
optimization, 
Journal of Machine Learning Research, 22 (2021), 1--34.

\end{thebibliography}
\end{document}